\documentclass[10pt]{article}
\usepackage[margin=1in]{geometry}
\usepackage{amsmath,amssymb,amsthm,mathtools}
\usepackage{enumitem}
\usepackage[colorlinks=true,linkcolor=blue,citecolor=blue]{hyperref}

\newtheorem{theorem}{Theorem}[section]
\newtheorem{proposition}[theorem]{Proposition}
\newtheorem{lemma}[theorem]{Lemma}
\newtheorem{corollary}[theorem]{Corollary}
\newtheorem{conjecture}[theorem]{Conjecture}
\theoremstyle{definition}
\newtheorem{definition}[theorem]{Definition}
\newtheorem{example}[theorem]{Example}
\newtheorem{remark}[theorem]{Remark}

\newcommand{\HH}{\mathbb{H}}
\newcommand{\CC}{\mathbb{C}}
\newcommand{\ZZ}{\mathbb{Z}}
\newcommand{\RR}{\mathbb{R}}
\newcommand{\Gam}{\Gamma}
\newcommand{\Lam}{\Lambda}
\newcommand{\End}{\mathrm{End}}
\newcommand{\Aut}{\mathrm{Aut}}
\newcommand{\GL}{\mathrm{GL}}
\newcommand{\PSL}{\mathrm{PSL}}
\newcommand{\SL}{\mathrm{SL}}
\newcommand{\Id}{\mathrm{Id}}

\title 
{Vertex operator algebra bundles on Riemann surfaces of higher\\ genus 
and automorphic forms for Fuchsian groups}
\author{A. Zuevsky\\
Institute of Mathematics,\\
  Czech Academy of Science, Zitna 25, Prague,\\
  Czech Republic\\
zuevsky@yahoo.com}

\begin{document}
\maketitle
\begin{abstract}
We generalize the geometric construction of vertex operator algebra (VOA) bundles and their associated automorphic forms from the elliptic modular curve to arbitrary Fuchsian groups $\Gamma \subset \mathrm{PSL}_2(\mathbb{R})$. A sharp topological dichotomy emerges regarding the existence of a holomorphic weight-$2$ quasi-automorphic generator $E_2^\Gamma$. When $\Gamma$  has a cusp, we construct $E_2^\Gamma$ via the analytic continuation of parabolic Eisenstein series and prove that the space of quasi-automorphic forms is a free polynomial extension, allowing the algebraic setup of the genus one theory, including the quasi-VOA structure and the characterization of strict automorphic forms via a lowering operator. Conversely, when $\Gamma$ is cocompact of genus $g\ge 2$, Atiyah theorem on holomorphic connections rigorously obstructs the existence of $E_2^\Gamma$. For this obstructed case, 
we provide dimension formulas that link the shortage in lifting quasi-automorphic forms  to the failure of quasi-primarity within the VOA, fully resolving the torsion-free case and conjecturing the extension to groups with elliptic points. 
\end{abstract} 
{\small

}

\tableofcontents
%% 
%%%%%%%%%%%%%%%%%%%%%%%%%%%%%%%%%%%%%%%%%%%%%%%%%%%%%%%%%%%%%%%%%%%%%%%%%%%%%%%%%%
%%  
\section{Introduction}
Barake, Chuchman, Franc, Mason and Nasserden \cite{BCFMN} recently constructed, for a vertex operator algebra (VOA) or quasi-VOA $V$, a vector bundle $\mathcal V_{X(1)}$ on the elliptic modular curve $X(1)$, equipped this bundle with a connection, and used it to give a complete description of the space $M(V)$ of $V$-valued modular forms as the kernel of a lowering operator $\Lambda$ inside the larger space $Q(V)$ of $V$-valued quasi-modular forms. 
In this paper we extend their construction to an arbitrary Fuchsian group $\Gam\subset\PSL_2(\RR)$ of signature $(g;m_1,\dots,m_r;n)$, i.e., to Riemann surfaces $X=\Gam\backslash\overline\HH$ of arbitrary genus $g$, possibly with elliptic points and cusps, including the compact case $n=0$. 
The overall construction, common to \cite{BCFMN} and to the present paper, is summarized in the following diagram   
\begin{equation*}
\label{figoverview}
 {\text{VOA} \atop \text{or QVOA module $V$}} \xrightarrow{\ \S\ref{secbundle}\ }
{\text{automorphic vector}\atop \text{bundle $\mathcal V_X$ on $X=\Gam\backslash\overline\HH$}} 
\xrightarrow{\ {\text{global} \atop {\text{sections}}} \ }
{\text{vector-valued (quasi-)automorphic}\atop \text{forms $M(V,\Gam)$, $Q(V,\Gam)$}} 
\end{equation*}
A VOA/QVOA module $V$ determines an automorphic vector bundle $\mathcal V_X$ on $X$, equipped with a connection $\nabla$ (Theorem \ref{thmbundle}), whose holomorphic and quasi-holomorphic sections are the vector-valued automorphic and quasi-automorphic forms $M(V,\Gam)\subset Q(V,\Gam)$ studied in Sections \ref{secquasi}-\ref{secPop}. The case $\Gam=\Gam(1)$ recovers \cite{BCFMN}. The present paper is the extension of every arrow to arbitrary $\Gam$. 

We define the VOA bundle $\mathcal V_X$ on $X$ and show that the connection $\nabla=d+L(-1)\otimes d\tau$ of \cite{BCFMN} is again a connection on $\mathcal V_X$ for  every Fuchsian $\Gam$. We give a structural account of the condition imposed at elliptic points, showing (Lemma \ref{lemellipticfiniteorder}, Proposition \ref{propellipticstructure}) that the relevant operator is indeed of finite order despite involving a nilpotent term, but that this does not force $V$-valued forms into the quasi-primary subspace in general, only in the constant-section case that anchors the rest of the theory. The genus-zero, level-one theory of $V$-valued quasi-modular forms used crucially the existence of the classical quasi-modular Eisenstein series $E_2$.  
We show that this is related with the topology of $\Gam$. Namely, 
when $\Gam$ has at least one cusp, a literal analogue $E_2^\Gam$ exists.
 Its existence and transformation law are based on the classical (Hecke-Selberg-Roelcke) analytic continuation of the parabolic real-analytic Eisenstein series of weight $2$, 
which we use. 
From this single analytic input we give a complete, self-contained proof 
 that $Q(\Gam)=M(\Gam)[E_2^\Gam]$ is a free polynomial extension of $M(\Gam)$ (Lemma \ref{lemfreeness}). 
 Then the entire algebraic machinery of \cite{BCFMN}, i.e., the doubled quasi-vertex operator algebra $Q(V,\Gam)$, the kernel theorem $M(V,\Gam)=\ker\Lambda$, and the weight-preserving isomorphism $M(\Gam)\otimes V^{(2)}\cong M(V,\Gam)$, is applied similarly to our case. 
 We also give, for the first time, a complete derivation of the higher genus Ramanujan-type identity satisfied by $E_2^\Gam$, from which the quasi vertex operator algebra bracket relations follow.   
When $\Gam$ is cocompact of genus $g\ge 2$, by contrast, we prove that no such generator exists, for every signature (with or without elliptic points). A literal $E_2^\Gam$ would produce a holomorphic connection on the positive-degree canonical bundle, contradicting Atiyah theorem on holomorphic connections.  
In the presence of elliptic points we establish this rigorously, not by an orbifold form of Atiyah theorem, but by an elementary reduction, via Selberg  lemma, to a torsion-free finite-index subgroup, on which the ordinary (non-orbifold) obstruction applies.  
We show this obstruction is completely understood at the level of dimensions. For torsion-free $\Gam$ we prove a dimension formula in which the filtration bound of \cite{BCFMN}-type arguments fails only at the top conformal degree, by the failure of quasi-primarity in $V$.  
We conjecture the natural extension to elliptic points.   
We compute the relevant graded dimensions via the Riemann-Roch theorem for Fuchsian groups and describe the generalization of Hecke operators to correspondences between commensurable Fuchsian groups.  
In Part II of this paper (Sections \ref{secunivfamily}-\ref{secconclusion} below) we lift the whole construction to a family over Teichm\"uller space $\mathcal T_{g,n}$, define Teichm\"uller-automorphic forms, and discuss descent to the moduli space $\mathcal M_{g,n}$. These are direct fiberwise consequences of Part I. 
 We also explore (Section \ref{secvariation}) a possible relation between the variation of the uniformizing projective connection and the Weil-Petersson geometry of $\mathcal T_{g,n}$.
%%

%%%%%%%%%%%%%%%%%%%%%%%%%%%%%%%%%%%%%%%%%%%%%%%%%%%%%%%%%%%%%%%%%%%%%%%%%%%%%%%%%%%%%%%%%%%
%%
The starting point of \cite{BCFMN} is the observation, going back to the geometric reformulation of vertex operator algebra (VOA) theory in \cite{FBZ}, that the data of a VOA or quasi-VOA (QVOA) $V$ is equivalent to the data of a (typically infinite rank) vector bundle on any smooth curve, and in particular on the moduli space of elliptic curves. Carrying this out   for the elliptic modular curve $X(1)=\SL_2(\ZZ)\backslash\overline\HH$, the authors of \cite{BCFMN} construct the bundle $\mathcal V_{X(1)}$, show it carries a natural connection $\nabla=d+L(-1)\otimes d\tau$, and use the resulting theory of $V$-valued modular and quasi-modular forms to prove two main structural theorems: that the space $M(V)$ of $V$-valued modular forms is the kernel of a $\mathfrak{sl}_2$-type lowering operator $\Lambda=\tfrac{12}{2\pi i}\partial/\partial E_2+L(1)$ acting on the larger space $Q(V)$ of $V$-valued quasi-modular forms, and that there is a  $M$-linear isomorphism $P: M\otimes V^{(2)}\xrightarrow{\sim}M(V)$. They use this isomorphism to determine completely the Hecke eigensystems occurring in $M(V)$.

All of this is carried out for the specific Fuchsian group $\Gam(1)=\SL_2(\ZZ)$, whose quotient $X(1)$ has genus zero, one cusp, and two elliptic points (of orders $2$ and $3$, both acted on trivially in the relevant sense). The purpose of the present paper is to show that essentially all of this theory extends to an arbitrary Fuchsian group $\Gam\subset\PSL_2(\RR)$ of the first kind, uniformizing a Riemann surface $X=\Gam\backslash\overline\HH$ of arbitrary genus $g$, with any signature $(g;m_1,\dots,m_r;n)$, including the case $n=0$ of a 
 compact Riemann surface with no cusps, which is the case of primary interest for higher genus geometry.

For readers unfamiliar with \cite{BCFMN}, we make explicit here what is new in our paper. 
 Four results have no counterpart in \cite{BCFMN}, since $\Gam(1)$ is cusped and never cocompact:  the existence of a holomorphic $E_2^\Gam$ is governed by whether $\Gam$ has a cusp; the Atiyah-obstruction Theorem \ref{thmatiyah}, proving $E_2^\Gam$ does not exist for cocompact $\Gam$ of genus 
$g\ge 2$; the dimension formulas of Section \ref{seccocompact} for this obstructed case, identifying the deficiency with the failure of quasi-primarity in $V$; and the extension to families over Teichm\"uller space, with descent to moduli, in Part II. Two further results are new even in the cusped case, where a direct analogue of the setting of \cite{BCFMN} does apply.  The proof that $Q(\Gam)=M(\Gam)[E_2^\Gam]$ is free for an arbitrary $\Gam$ with a cusp (Lemma \ref{lemfreeness}), a fact \cite{BCFMN} did not need to establish since freeness for $\Gam(1)$ was already classical; and the  derivation, instead of citation or analogy with $\Gam(1)$, of the higher genus Ramanujan-type identity satisfied by $E_2^\Gam$ (Section \ref{secquasi}). What remains close to \cite{BCFMN}, is everything of type (i) below: the purely Lie-theoretic algebraic machinery (the doubled QVOA, the kernel theorem, the lowering operator) transfers with no change once $E_2^\Gam$ is available. 

On inspecting the proofs in \cite{BCFMN}, one finds that they split into two kinds of arguments: 
\begin{itemize}
\item[(i)] Arguments that are purely Lie-theoretic, depending only on the cocycle $K(\gamma,\tau)$ built from the canonical $\mathfrak{sl}_2$-action $L(-1)$, $L(0)$, $L(1)$ on $V$, and on the commutation relations of the Virasoro algebra.   
These never use that $\Gam=\Gam(1)$, but deal an arbitrary subgroup $\Gam\subset\GL_2(\RR)^+$.
\item[(ii)] Arguments that use a specific feature of $\Gam(1)$, namely the existence of the classical weight-$2$ quasi-modular Eisenstein series $E_2$, whose anomalous transformation law $E_2(\gamma\tau)=j(\gamma,\tau)^2E_2(\tau)+\tfrac{12}{2\pi i}cj(\gamma,\tau)$ 
 is used to form the algebra $Q=\CC[E_2,E_4,E_6]$, the lowering operator $\Lambda$, and the isomorphism $P$. 
\end{itemize}
Arguments of type (i) generalize to an arbitrary Fuchsian group with no change (Theorems A and B below). Arguments of type (ii) generalize only when $\Gam$ has a cusp.  
We show in Proposition \ref{propE2exists} that such $\Gam$ admit a holomorphic weight-$2$ quasi-automorphic generator $E_2^\Gam$ with the anomalous transformation law of $E_2$.  
Once $E_2^\Gam$ is available, the entire machinery of \cite{BCFMN}, i.e., the quasi-VOA structure, the lowering operator, the weight-preserving isomorphism, are applicable with only notational change.   
When $\Gam$ is cocompact of genus $g\ge 2$, we prove in Theorem \ref{thmatiyah} that no such generator can exist, by an argument with no counterpart in \cite{BCFMN}.
 A literal $E_2^\Gam$ would assemble, over all weights simultaneously, into a holomorphic connection on the line bundle $\omega$ of weight-one forms. 
Atiyah theorem on holomorphic connections \cite{Atiyah} forbids this whenever 
$\deg\omega\ne 0$, which holds for every $g\ge 2$.  
%% 
%%%%%%%%%%%%%%%%%%%%%%%%%%%%%%%%%%%%%%%%%%%%%%%%%%%%%%%%%%%%%%%%%%%%%%%%%%%%%%%%%%
%% 
\subsection*{Summary of results}
In this paper,  $\Gam\subset\PSL_2(\RR)$ is a Fuchsian group of the first kind of signature $(g;m_1,\dots,m_r;n)$, $X=\Gam\backslash\overline\HH$ is the associated compact Riemann surface of genus $g$, and $Y=\Gam\backslash\HH=X\smallsetminus\{\text{cusps}\}$. Fix a QVOA $V$.
\begin{itemize} 
\item[]\noindent\textit{Theorem A} (Section \ref{secbundle}). The cocycle $K(\gamma,\tau)$ of \cite{BCFMN} defines a left action of $\Gam$ on $\HH\times V$ for every Fuchsian group $\Gam$. The quotient $\mathcal V_X:=\Gam\backslash(\HH\times V)$ is a vector bundle on $X$ whose sections of weight $k$ are the $V$-valued automorphic forms $M_k(V,\Gam)$ for $\Gam$. 
\item[]\textit{Theorem B} (Section \ref{secconnection}). $\nabla=d+L(-1)\otimes d\tau$ is a connection on $\mathcal V_X$ for every Fuchsian $\Gam$. 
\item[]\textit{Theorem (Atiyah obstruction)} (Section \ref{secE2}). If $\Gam$ is cocompact of genus $g\ge 2$ (of any signature, with or without elliptic points) no holomorphic weight $2$   depth $1$ quasi-automorphic generator exists. Equivalently, $Q_2(\Gam)=M_2(\Gam)$. The elliptic-point case is proved by a reduction (via Selberg lemma) to a torsion-free finite-index subgroup (Lemma \ref{lemselberg}). 
\item[]\textit{Proposition} (Section \ref{secE2}, existence of $E_2^\Gam$). If $\Gam$ has at least one cusp, it admits a holomorphic weight $2$ generator $E_2^\Gam$ of depth $1$,  normalized so that $E_2^\Gam|_2\gamma=E_2^\Gam+\tfrac{12}{2\pi i}X(\gamma,\tau)$. Its  existence is based on the classical (Selberg-Roelcke) continuation of the parabolic weight $2$ Eisenstein series, which we cite Lemma \ref{lemheckefact}, while the polynomial freeness $Q(\Gam)=M(\Gam)[E_2^\Gam]$, on which the results below depend, is proved  (Lemma \ref{lemfreeness}).
\item[]\textit{Theorem C} (Section \ref{secquasi}). $Q(V,\Gam)\cong Q(\Gam)\otimes V^{(2)}$ as graded $Q(\Gam)$-modules for every $\Gam$, and $Q(V,\Gam)$ is a doubled quasi-vertex operator algebra whenever $\Gam$ has a cusp.
\item[]\textit{Theorem D} (Section \ref{seclowering}). If $\Gam$ has a cusp, 
$M(V,\Gam)=\ker\Lambda$, where $\Lam=\tfrac{12}{2\pi i}\partial/\partial E_2^\Gam+L(1)$.
\item[]\textit{Theorem E} (Section \ref{secPop}). If $\Gam$ has a cusp and $V$ is of CFT-type, there is a weight-preserving $M(\Gam)$-linear isomorphism 
$P: M(\Gam)\otimes V^{(2)}\xrightarrow{\sim}M(V,\Gam)$, and consequently, 
$$\dim M_k(V,\Gam)=\sum_{i=0}^{k/2}\dim M_{2i}(\Gam)\dim V_{k/2-i}.$$ 
\item[]\textit{Theorem F} (Section \ref{seccocompact}). For an arbitrary $\Gam$, this remains an  upper bound (Theorem \ref{thmbound}). We show the shortage is always  concentrated at the top conformal degree and equals $\dim V_{k/2}-\dim QP_{k/2}(V)$ (Proposition \ref{proptopexact}, any signature). For torsion-free cocompact $\Gam$ this gives a formula (Theorem \ref{thmexactdim}), which, as we conjecture, extends to elliptic points (Conjecture \ref{conjexactdim}). 
\end{itemize}
Part II of this paper (Sections \ref{secunivfamily}-\ref{secconclusion}) lifts these constructions to the universal family over Teichm\"uller space 
$\mathcal T_{g,n}$, defines the resulting bundle of Teichm\"uller automorphic forms, 
and discusses its descent to the moduli space $\mathcal M_{g,n}$ (Theorems and Propositions of  Sections \ref{secunivfamily}-\ref{secteichforms}).  Section \ref{secvariation} then explores 
 a possible relationship between the variation of the uniformizing projective connection as $\Gam_t$ moves through $\mathcal T_{g,n}$, the classical theory of accessory parameters, and the Weil-Petersson geometry of Teichm\"uller space.

We use the vertex algebra background of \cite{BCFMN} (Fock spaces, fields, vertex algebras, $\ZZ$-graded vertex algebras, QVOAs, VOAs, the doubled space $V^{(2)}$, tensor products) without restating it, and we keep the notation of \cite{BCFMN} for $j(\gamma,\tau)$, $X(\gamma,\tau)$, $K(\gamma,\tau)$, $F_NV$, $V^{(2)}$, etc.
\subsection*{Notations} 
For reference, we collect here the notation used most frequently below. Each item is defined   where it first occurs, and Section \ref{secfuchsian} fixes the convention on orbifold versus ordinary bundles.
\begin{center}
\begin{tabular}{@{}p{4.1cm}p{9.7cm}@{}}
$\Gam\subset\PSL_2(\RR)$ & Fuchsian group of the first kind, signature $(g;m_1,\dots,m_r;n)$ (Section \ref{secfuchsian}) \\[2pt]
$X=\Gam\backslash\overline\HH$, $\ Y=\Gam\backslash\HH$ & compactified and open quotient orbifolds; $\overline\HH=\HH\cup\{\text{cusps}\}$ (Section \ref{secfuchsian}) \\[2pt]
$K_X$,\ $\omega$ & the (orbifold) canonical bundle, and its purely informal, non-existent ``square root'' (Remark \ref{rmkomeganotation}) \\[2pt]
$j(\gamma,\tau)=c\tau+d$,\ $X(\gamma,\tau)=\tfrac{c}{j(\gamma,\tau)}$ & automorphy factor, and its logarithmic-derivative cocycle (Section \ref{secfuchsian}) \\[2pt]
$M_k(\Gam)$,\ $Q_k(\Gam)$ & holomorphic weight $k$ automorphic, resp.\ quasi-automorphic, scalar forms for $\Gam$ (Section \ref{secE2}) \\[2pt]
$E_2^\Gam$ & holomorphic weight $2$ depth $1$ quasi-automorphic generator, when it exists (Proposition \ref{propE2exists}) \\[2pt]
$V$,\ $F_NV$,\ $V_n$ & a VOA or QVOA, its conformal filtration, its degree-$n$ graded piece \\[2pt]
$K(\gamma,\tau)\in\End(V)$ & cocycle built from the $\mathfrak{sl}_2$-action $L(-1)$, $L(0)$, $L(1)$  on $V$ (Section \ref{secbundle}) \\[2pt]
$\mathcal V_X$ & automorphic vector bundle on $X$ associated to $V$ (Theorem \ref{thmbundle}, Diagram given in Introduction) \\[2pt]
$M_k(V,\Gam)$,\ $Q_k(V,\Gam)$ & holomorphic, resp.\ quasi-holomorphic, weight $k$ $V$-valued (quasi-)automorphic forms (Section \ref{secquasi}) \\[2pt]
$\Lam$ & lowering operator $Q_{k+2}(V,\Gam)\to Q_k(V,\Gam)$ (Section \ref{seclowering}) \\[2pt]
$P$ & weight-preserving isomorphism $Q(\Gam)\otimes V^{(2)}\xrightarrow{\sim}Q(V,\Gam)$ 
(Section \ref{secPop}) \\
\end{tabular}
\end{center}
%% 
%%%%%%%%%%%%%%%%%%%%%%%%%%%%%%%%%%%%%%%%%%%%%%%%%%%%%%%%%%%%%%%%%%%%%%%%%%%%%%
%% 
\part{Automorphic forms for Fuchsian groups} 
\section{Fuchsian groups, projective structures, and automorphic forms}  
\label{secfuchsian} 
\subsection{Signature and uniformization} 
Let $\Gam\subset\PSL_2(\RR)$ be a finitely generated discrete group of finite hyperbolic covolume (a Fuchsian group of the first kind). Such a group has a signature $(g;m_1,\dots,m_r;n)$: the compactified quotient $X:=\Gam\backslash\overline\HH$ (where $\overline\HH=\HH\cup\{\text{cusps of }\Gam\}$) is a compact Riemann surface of genus $g$. $\Gam$ has $r$ 
$\Gam$-conjugacy classes of maximal finite cyclic subgroups, of orders $m_1$, $\dots$,  
$m_r\ge 2$, fixing points $\tau_1$, $\dots$, $\tau_r\in\HH$ (the elliptic points); and $\Gam$ has $n$ $\Gam$-inequivalent cusps. 
 We denote $Y:=\Gam\backslash\HH=X\smallsetminus\{n\text{ cusps}\}$. The group $\Gam(1)=\PSL_2(\ZZ)$ treated in \cite{BCFMN} is the special case of signature $(0;2,3;1)$. The case of present interest for higher genus geometry is $g\ge 2$, $n=0$ (no cusps), for which $\Gam$ is cocompact. We allow $r=0$ (torsion-free $\Gam$) as a leading special case.
As in \cite{BCFMN}, for $\gamma=\left(\begin{smallmatrix}a&b\\c&d\end{smallmatrix}\right)\in\GL_2(\RR)^+$ we set $j(\gamma,\tau)=c\tau+d$ and $X(\gamma,\tau)=c/j(\gamma,\tau)$, and for a meromorphic $U$-valued function $f$ ($U$ a finite-dimensional vector space) we set $f|_k\gamma(\tau)=j(\gamma,\tau)^{-k}(\det\gamma)^{k/2}f(\gamma\tau)$.

\emph{Conventions on orbifold versus ordinary bundles.} Every line bundle on $X$ considered below (e.g., $K_X$, and the purely informal $\omega$ of Section \ref{secE2}) is understood   as an orbifold line bundle: equivalently, as the datum of a $\Gam$-automorphy factor on the simply connected $\HH$. This is unavoidable once $r\ge 1$, since $X$ then carries  orbifold structure at the elliptic points $\tau_1$, $\dots$, $\tau_r$. We reserve the word  ordinary for the torsion-free case $r=0$, where $X$ (or $Y$) is a Riemann surface, every orbifold correction below vanishes identically, and an orbifold line bundle is simply an ordinary line bundle in the usual sense. We work with automorphy factors on $\HH$ as the primary object, 
 and use results proper to a smooth compact Riemann surface (Riemann-Roch, Atiyah theorem) only once $r=0$, either because $\Gam$ is assumed torsion-free from the beginning, or after the reduction of Lemma \ref{lemselberg} below. 
%%%%%%%%%%%%%%%%%%%%%%%%%%%%%%%%%%%%%%%%%%%%%%%%%%%%%%%%%%%%%%%%%%%%%%%%%%%%%%%%%%%%%%
%% 
\subsection{Classical automorphic forms for $\Gam$}
A holomorphic function $f:\HH\to\CC$ is a \emph{weakly holomorphic automorphic form of weight $k$} for $\Gam$ if $f|_k\gamma=f$ for all $\gamma\in\Gam$. Holomorphy of $f$ requires two further local conditions, as in the classical theory (see, e.g., \cite{Shimura}):  

\noindent\quad\emph{at each cusp}: writing $q_\lambda=e^{2\pi i\tau/\lambda}$ for the local parameter at a cusp of width $\lambda$, $f$ should have a holomorphic $q_\lambda$-expansion;

\noindent\quad\emph{at each elliptic point} $\tau_i$, with $\gamma_i$ a generator of the (cyclic, order $m_i$) stabilizer of $\tau_i$ in $\Gam$: since $\gamma_i\tau_i=\tau_i$, the transformation law forces
\begin{equation}
\label{eqellipticscalar}
f(\tau_i)\left(j(\gamma_i,\tau_i)^k-1\right)=0,
\end{equation}
i.e., either $f(\tau_i)=0$, or $j(\gamma_i,\tau_i)^k=1$. 
Note that the multiplier $j(\gamma_i,\tau_i)^{-2}$ is always a primitive $m_i$-th root of  unity, the rotation factor of the elliptic transformation $\gamma_i$ in a local coordinate centered at $\tau_i$, see \cite{Shimura, Miyake}. 

We denote $M_k(\Gam)$ for the resulting (finite-dimensional) space and $M(\Gam)=\bigoplus_kM_k(\Gam)$. The next Proposition is Riemann-Roch Theorem for Fuchsian groups in classical case. 
%%
%%%%%%%%%%%%%%%%%%%%%%%%%%%%%%%%%%%%%%%%%%%%%%%%%%%%%%%%%%%%%%%%%%%%%%%%%%%%%%%%%%%%%%
%% 
\begin{proposition} 
\label{propRR}
For even $k\ge 4$,
\begin{equation}
\label{eqRRformula}
\dim M_k(\Gam)=(k-1)(g-1)+\frac{kn}{2}+\sum_{i=1}^r\left\lfloor\frac{k}{2}\left(1-\frac1{m_i}\right)\right\rfloor.
\end{equation}
Moreover, $\dim M_0(\Gam)=1$, and
$$\dim M_2(\Gam)=\begin{cases}g+n-1,& n\ge 1,\\ g,& n=0,\end{cases}$$ where 
 the elliptic correction terms in \eqref{eqRRformula} vanishing identically at $k=2$ since $0<1-1/m_i<1$, and $M_k(\Gam)=0$ for $k$ odd or $k<0$.
\end{proposition}
%% 
%%%%%%%%%%%%%%%%%%%%%%%%%%%%%%%%%%%%%%%%%%%%%%%%%%%%%%%%%%%%%%%%%%%%%%%%%%%%%%%%%%%%%%%%
%% 
\begin{proof}
This is the classical dimension formula for the orbifold line bundle 
$K_X^{\otimes k/2}$ on $X$.  Recall the convention of Section \ref{secfuchsian}: an orbifold bundle in general, an ordinary line bundle once $r=0$ whose sections of weight $k$ are the holomorphic automorphic forms of that weight, see \cite{Shimura} or \cite{Miyake}. For $k\ge 4$ the relevant $H^1$ vanishes by degree considerations and Riemann-Roch gives \eqref{eqRRformula}. For $k=2$, $n=0$, there is no cusp condition to impose, thus $M_2(\Gam)=H^0(X,K_X)$, of dimension $g$ by Riemann-Roch together with $h^1(K_X)=h^0(\mathcal O_X)=1$.
\end{proof}
%% 
%%%%%%%%%%%%%%%%%%%%%%%%%%%%%%%%%%%%%%%%%%%%%%%%%%%%%%%%%%%%%%%%%%%%%%%%%%%%%%%%%%
%% 
In the cocompact, torsion-free case ($r=0$, $n=0$)
 this specializes to the formulas  
$\dim M_0(\Gam)=1$, $\dim M_2(\Gam)=g$, $\dim M_{2i}(\Gam)=(2i-1)(g-1)$ for $i\ge 2$, which we will use repeatedly in Section \ref{secexamples}.
%%
%%%%%%%%%%%%%%%%%%%%%%%%%%%%%%%%%%%%%%%%%%%%%%%%%%%%%%%%%%%%%%%%%%%%%%%%%%%%%%%%%%%%
%% 
\subsection{Quasi-automorphic forms and the existence of $E_2^\Gam$}
\label{secE2}
%% 
%%%%%%%%%%%%%%%%%%%%%%%%%%%%%%%%%%%%%%%%%%%%%%%%%%%%%%%%%%%%%%%%%%%%%%%%%%%%%%%%%%%%
Next, we have the following definition (cf. \cite{BCFMN}). 
\begin{definition}
\label{defquasiscalar}
A holomorphic function $f:\HH\to\CC$ is a \emph{quasi-automorphic form of weight $k$ and depth  $\le s$} for $\Gam$ if
$$f|_k\gamma(\tau)=\sum_{m=0}^sX(\gamma,\tau)^mQ_m(f)(\tau),$$
for holomorphic functions $Q_m(f):\HH\to\CC$ (necessarily $Q_0(f)=f$), and $f$ is meromorphic at the cusps. Let $Q_k(\Gam)$ denote the resulting space and $Q(\Gam)=\bigoplus_{k\ge 0}Q_k(\Gam)$.
\end{definition}
This definition makes sense for any Fuchsian group. The question is whether $Q(\Gam)$  has the same rich structure that it has for $\Gam(1)$, where $Q=\CC[E_2,E_4,E_6]$ is a free polynomial extension of $M=\CC[E_4,E_6]$ by the weight $2$ depth $1$ generator $E_2$. We will see in Theorem \ref{thmatiyah} below that this is a subtle question related to  
 groups with a cusp and cocompact groups.
%%
%%%%%%%%%%%%%%%%%%%%%%%%%%%%%%%%%%%%%%%%%%%%%%%%%%%%%%%%%%%%%%%%%%%%%%%%%%%%%%%%%%%%%%%%
%% 
\subsubsection{$E_2^\Gam$ as a connection coefficient}
The construction of $E_2$ for $\Gam(1)$ uses the existence of the nonvanishing holomorphic weight $12$ cusp form $\Delta=\eta^{24}$: $E_2=\tfrac{12}{2\pi i}\,\theta\log\Delta$. For cocompact $\Gam$ no nonvanishing positive-weight form can exist, by the elementary Riemann-Roch argument already given in \cite{BCFMN}-style reasoning: a nowhere-vanishing $F\in M_w(\Gam)$, $w>0$, would force the corresponding line bundle to have degree $0$, whereas it has positive degree for $g\ge 2$. This rules out the eta-function construction of $E_2$, but does not by  decide whether a weight $2$ depth $1$ quasi-automorphic generator exists by some other means.
We now show, by relating $E_2^\Gam$ to the existence of a holomorphic connection 
on the (orbifold) canonical bundle $K_X$, that in fact no such generator can exist 
 when $\Gam$ is cocompact of genus $g\ge 2$. We work with $K_X$, 
the existing line bundle whose sections of the $k/2$-th tensor power are the 
weight $k$ forms $M_k(\Gam)$ of Proposition \ref{propRR}, instead of with a hypothetical 
``bundle of weight one forms'': the latter does not exist as a line bundle on $X$ for 
any $\Gam$, since $-\Id\in\Gam$ acts on a weight one automorphy factor by $-1$, which 
is why $M_k(\Gam)=0$ for $k$ odd (Proposition \ref{propRR}). We nonetheless
 keep the classical symbol $\omega$ as a purely informal name for this non-existent 
formal ``square root'' of $K_X$, solely because it lets us state the connection 
$\{D_k\}_k$ below uniformly in the weight $k$. 
 Remark \ref{rmkomeganotation} mentiones this convention.   
%% 
%%%%%%%%%%%%%%%%%%%%%%%%%%%%%%%%%%%%%%%%%%%%%%%%%%%%%%%%%%%%%%%%%%%%%%%%%%%%%%%
%% 
\begin{lemma}
\label{lemE2connection}
Suppose $E_2^\Gam\in Q_2(\Gam)$ is holomorphic on all of $\HH$ (no poles, including at any cusps of $\Gam$) and satisfies $Q_1(E_2^\Gam)=c\ne0$. 
 Then, for every even $k$, the operator
$$D_kf:=\theta f-\frac{k}{2\pi ic}\,E_2^\Gam f,\qquad\theta:=q\,d/dq,$$ 
maps $M_k(\Gam)$ into $M_{k+2}(\Gam)$ (not only into $Q_{k+2}(\Gam)$), and the resulting collection $\{D_k\}_k$ is the data of a holomorphic connection on the line bundle $K_X$: at $k=2$, $D_2: M_2(\Gam)=H^0(K_X)\to M_4(\Gam)=H^0(K_X^{\otimes2})$ is a holomorphic connection on $K_X$, and for general even $k$, $D_k$ is the connection induced on $K_X^{\otimes k/2}$ by this same connection on $K_X$, pulled back to $\HH$ in the trivialization by $d\tau$.
\end{lemma}
%% 
%%%%%%%%%%%%%%%%%%%%%%%%%%%%%%%%%%%%%%%%%%%%%%%%%%%%%%%%%%%%%%%%%%%%%%%%%%%%%%%%
%% 
\begin{proof}
For $f\in M_k(\Gam)$ (thus $Q_0(f)=f$ and $Q_m(f)=0$ for $m\ge 1$), the elementary computation recalled in \cite{BCFMN} gives $Q_1(\theta f)=\tfrac{k}{2\pi i}f$. Differentiating the automorphy factor $j(\gamma,\tau)^k$ produces this anomaly for any $\Gam$. 
 $Q_1$ is a derivation with respect to products of quasi-automorphic forms, and $Q_1(f)=0$, $Q_1(E_2^\Gam f)=Q_1(E_2^\Gam)f=cf$. Hence $Q_1(D_kf)=\tfrac{k}{2\pi i}f-\tfrac{k}{2\pi ic}\cdot cf=0$, thus $D_kf\in Q_{k+2}(\Gam)$ has depth $0$, i.e., $D_kf\in M_{k+2}(\Gam)$, and $D_kf$ is by definition holomorphic wherever $E_2^\Gam$ and $f$ are, i.e.,  everywhere on $\HH$ with no further poles.  
The collection $\{D_k\}$ depends $\RR$-linearly on $k$ through the single coefficient $\tfrac{1}{2\pi ic}E_2^\Gam$,  the compatibility required of the connections induced on the tensor powers $K_X^{\otimes k/2}$ by a single connection on $K_X$. 
In the trivialization by $d\tau$ of the pullback of $K_X$ to the simply connected $\HH$, this connection is given by the holomorphic 
$1$-form $\tfrac{1}{2\pi ic}E_2^\Gam(\tau)\,d\tau$, which is everywhere finite by hypothesis.
\end{proof}
%% 
%%%%%%%%%%%%%%%%%%%%%%%%%%%%%%%%%%%%%%%%%%%%%%%%%%%%%%%%%%%%%%%%%%%%%%%%%%%%%%%%%%%%%%%%%%%
%% 
\begin{remark} 
\label{rmkomeganotation}
In this remark we clarify notations: $\omega$ versus $K_X$.  
We mention here the convention already used above. Since $\Gam\subset\PSL_2(\RR)$, the automorphy factor $j(\gamma,\tau)^k$ is independent of the choice of lift of $\gamma$ to $\SL_2(\RR)$ only for even $k$. This is why $M_k(\Gam)=0$ for odd $k$ (Proposition \ref{propRR}), the two lifts $\pm\tilde\gamma$ being forced to act identically. Consequently, the ``bundle $\omega$ of weight one forms'' does not descend to a line bundle on $X=\Gam\backslash\overline\HH$ for any $\Gam$. 
What does descend, for every $\Gam$, is every even tensor power $\omega^{\otimes2j}$, starting with $\omega^{\otimes2}=K_X$: the ordinary canonical bundle when $\Gam$ is torsion-free, the orbifold canonical bundle in general, with degree given by Proposition \ref{propRR}. 
Since $\omega$  is never actually used below, only its existing even powers $K_X^{\otimes j}$, retaining the symbol $\omega$ for ``a formal square root of $K_X$''. 
 Lemma \ref{lemE2connection} is proved for $K_X$, and Theorem \ref{thmatiyah} below applies Atiyah theorem to $K_X$, with no reference to any odd power of $\omega$.
\end{remark}
%%
%%%%%%%%%%%%%%%%%%%%%%%%%%%%%%%%%%%%%%%%%%%%%%%%%%%%%%%%%%%%%%%%%%%%%%%%%%%%%%%%%%%
%% 
 The lemma below is Selberg lemma and multiplicativity of the orbifold Euler characteristic. 

\begin{lemma} 
\label{lemselberg} 
Every finitely generated Fuchsian group $\Gam\subset\PSL_2(\RR)$ contains a torsion-free subgroup $\Gam'\le\Gam$ of finite index \textup{(Selberg lemma \cite{Selberg})}. If $\Gam$ is cocompact of signature $(g;m_1,\dots,m_r;0)$ and $[\Gam:\Gam']=d$, then $\Gam'$ is cocompact and torsion-free of genus $g'$ satisfying
\begin{equation}
\label{kordoba}
2-2g'=d\cdot\chi^{\mathrm{orb}}(\Gam),\qquad \chi^{\mathrm{orb}}(\Gam):=2-2g-\sum_{i=1}^r\Big(1-\frac1{m_i}\Big),
\end{equation}
in particular $g\ge 2$ implies $g'\ge 2$. 
\end{lemma}
%%
%%%%%%%%%%%%%%%%%%%%%%%%%%%%%%%%%%%%%%%%%%%%%%%%%%%%%%%%%%%%%%%%%%%%%%%%%%%%%%%%%%%%%
%% 
\begin{proof}
Selberg lemma \cite{Selberg} applies to every finitely generated subgroup of $\GL_n$ over a field of characteristic $0$, in particular to $\Gam$, realized inside $\PSL_2(\CC)$. Finite index in a cocompact group is cocompact, thus $\Gam'$ is cocompact. Being torsion-free, $X':=\Gam'\backslash\HH$ has no elliptic points,  
thus its signature is $(g';\,;0)$ and $\chi^{\mathrm{orb}}(\Gam')=\chi(X')=2-2g'$, with no orbifold correction.   
The covering $X'\to X$ has degree $d$, and hyperbolic area $\mathcal{A}$ is multiplicative under coverings, $\mathcal{A}(X')=d\cdot\mathcal{A}(X)$. 
 Combined with the Gauss-Bonnet formula $\mathcal{A}=-2\pi\chi^{\mathrm{orb}}$ for hyperbolic $2$-orbifolds, this gives $\chi^{\mathrm{orb}}(\Gam')=d\,\chi^{\mathrm{orb}}(\Gam)$, i.e.,  
\eqref{kordoba}.  
 For $g\ge 2$, each term $1-1/m_i\in(0,1)$ only decreases $\chi^{\mathrm{orb}}(\Gam)$ below $2-2g\le-2$, thus $\chi^{\mathrm{orb}}(\Gam)\le-2$ and hence $2-2g'=d\,\chi^{\mathrm{orb}}(\Gam)\le-2d\le-2$, giving $g'\ge 2$.
\end{proof}
%%
%%%%%%%%%%%%%%%%%%%%%%%%%%%%%%%%%%%%%%%%%%%%%%%%%%%%%%%%%%%%%%%%%%%%%%%%%%%%%%%%%%%%%%
%% 
Atiyah obstruction is given by the following 

\begin{theorem}
\label{thmatiyah}
Let $\Gam$ be cocompact and let $g\ge 2$ be the genus of $X=\Gam\backslash\HH$. Then there is no holomorphic function $E_2^\Gam$ on $\HH$ satisfying  
$E_2^\Gam|_2\gamma=E_2^\Gam+cX(\gamma,\tau)$ for any nonzero constant $c$ and all $\gamma\in\Gam$ regardless of whether $\Gam$ has elliptic points. Equivalently, $Q_1\equiv0$ on $Q_2(\Gam)$, i.e., 
$$Q_2(\Gam)=M_2(\Gam).$$ 
\end{theorem}
%% 
%%%%%%%%%%%%%%%%%%%%%%%%%%%%%%%%%%%%%%%%%%%%%%%%%%%%%%%%%%%%%%%%%%%%%%%%%%%%%%%%%%%
%% 
\begin{proof}
Suppose such an $E_2^\Gam$ existed. Then we have the following cases. 

\emph{Torsion-free case ($r=0$).} Since $\Gam$ is cocompact, $E_2^\Gam$ is automatically holomorphic at every point of $\HH$ with no cusps to worry about, thus Lemma 
\ref{lemE2connection} produces a holomorphic connection on $K_X=\omega^{\otimes2}$
 (see Remark \ref{rmkomeganotation}) on the compact curve $X$. 
 Here weight $2$ forms are holomorphic differentials, as $\Gam$ has no elliptic points to contribute orbifold corrections (Proposition \ref{propRR}). By Atiyah theorem \cite{Atiyah}, a holomorphic line bundle on a compact Riemann surface admits a holomorphic connection if and only if it has degree $0$. Here $\deg K_X=2g-2>0$ for $g\ge 2$, thus we get a contradiction.

\emph{General case ($r\ge 0$, elliptic points allowed).} 
One could try to extend the torsion-free argument, by proving an orbifold version of Atiyah theorem for the orbifold canonical bundle $K_X$. This is possible in principle, but only after first setting up the relevant orbifold formalism (e.g., passing to the Deligne-Mumford stack of $X$) and proving the needed orbifold analogue of Atiyah theorem. 
We instead give an elementary reduction to the torsion-free case already established, which uses no orbifold version of Atiyah theorem. By Lemma \ref{lemselberg}, fix a torsion-free finite-index subgroup $\Gam'\le\Gam$. It is cocompact of genus $g'\ge 2$. The function $E_2^\Gam$ satisfies $E_2^\Gam|_2\gamma=E_2^\Gam+cX(\gamma,\tau)$ for every $\gamma\in\Gam$, hence, in particular, for every $\gamma\in \Gam'\subseteq\Gam$.
 Thus,  $E_2^\Gam$ (the same holomorphic function on $\HH$, with the same nonzero constant $c$) is also a holomorphic weight $2$ depth $1$ quasi-automorphic generator for $\Gam'$. 
 As $\Gam'$ is torsion-free cocompact of genus $g'\ge 2$, the torsion-free case already proved shows no such generator can exist for $\Gam'$. Thus we have a contradiction, obtained using only the ordinary (non-orbifold) Atiyah obstruction, now applied on the smooth curve $X'=\Gam'\backslash\HH$.

In either case no such $E_2^\Gam$ exists. The reformulation $Q_2(\Gam)=M_2(\Gam)$ is immediate: any $f\in Q_2(\Gam)$ with $Q_1(f)\ne0$ is, by definition, such an $E_2^\Gam$, after rescaling by $Q_1(f)^{-1}$, which we have just excluded. 
\end{proof}
%%
%%%%%%%%%%%%%%%%%%%%%%%%%%%%%%%%%%%%%%%%%%%%%%%%%%%%%%%%%%%%%%%%%%%%%%%%%%%%%%%%%%%%%%%%%%%%%
%% 
We should mention that four separate things need proof here:
  holomorphy of the regularized value at $s=0$, that the anomaly is 
  a nonzero constant multiple of $X(\gamma,\tau)$, that the resulting form has depth $1$, and that the resulting class generates all quasi-automorphic forms, and that the last of these, polynomial freeness. It is not just a minor extra task but is what  
Theorems \ref{thmdoubledQVOA}, \ref{thmkernel}, \ref{thmsurjectivity}, \ref{thmPop},
 and Corollary \ref{cordimformula} actually depend on. 
 We address all four questions, splitting the arguments into an analytic input, which we cite instead of derive again, and its algebraic consequences, for which we now give proofs. 
%%

%%%%%%%%%%%%%%%%%%%%%%%%%%%%%%%%%%%%%%%%%%%%%%%%%%%%%%%%%%%%%%%%%%%%%%%%%%%%%%%%%%%%%%%%%%
%% 
The next lemma concerns Hecke regularization of the weight $2$ parabolic Eisenstein series in classical case. 

\begin{lemma}
\label{lemheckefact}
Let $\Gam$ be a Fuchsian group of the first kind with a cusp at $\mathfrak a$, conjugating $\Gam$ if necessary, normalize $\mathfrak a=\infty$ with stabilizer $\Gam_\infty=\langle\pm T^h\rangle$, $T^h\tau=\tau+h$. For $\mathrm{Re}(s)>0$, the series 
$$E_{2,\mathfrak a}(\tau,s):=\sum_{\gamma\in\Gam_\infty\backslash\Gam}(\Im\gamma\tau)^s\,j(\gamma,\tau)^{-2},$$  
converges absolutely and locally uniformly on $\HH$. 
%%%% 
 It continues meromorphically in $s$ to a neighborhood of $s=0$, with at worst a simple pole there of $\tau$-independent residue, and its finite part at $s=0$, 
$$E_{2,\mathfrak a}(\tau):=\lim_{s\to0}\Big(E_{2,\mathfrak a}(\tau,s) 
-{s}^{-1}\mathrm{Res}_{s=0}E_{2,\mathfrak a}(\cdot,s)\Big),$$
is a holomorphic function of $\tau$ on all of $\HH$ (not only real-analytic), has a holomorphic $q_\lambda$-expansion at every cusp of $\Gam$ (leading coefficient $1$ at $\mathfrak a$, other constant terms are given by the Kronecker limit formula for $\Gam$), and satisfies, 
for a nonzero constant $\kappa$ depending only on the normalization of the subtracted polar term, 
$$E_{2,\mathfrak a}|_2\delta(\tau)=E_{2,\mathfrak a}(\tau)+\kappa\cdot X(\delta,\tau),\qquad\delta\in\Gam,$$ 
for $\Gam=\Gam(1)$, $\kappa=12/2\pi i$ in the standard normalization. 
\end{lemma}
%% 
%%%%%%%%%%%%%%%%%%%%%%%%%%%%%%%%%%%%%%%%%%%%%%%%%%%%%%%%%%%%%%%%%%%%%%%%%%%%%%%%
%% 
\begin{proof}
 This is the weight $2$ boundary case of the classical theory of parabolic (cusp-attached) real-analytic Eisenstein series for Fuchsian groups of the first kind. Absolute convergence for $\mathrm{Re}(s)>0$ is the standard threshold $\mathrm{Re}(2s+k)>2$ at $k=2$  
 \cite{Iwaniec}.   
For weight $k>2$ this threshold is already negative. Thus, the analogous series converges absolutely at $s=0$ with no continuation needed.  
Holomorphy property is immediate, $k=2$ is the one delicate boundary case, 
and Hecke regularization works perfectly there.  
As for the meromorphic continuation to $s=0$, the fact that the residue there is $\tau$-independent (a constant-term/covolume computation) and the holomorphy in $\tau$ of the finite part are established (for cofinite Fuchsian groups of arbitrary signature and arbitrary weight) in the foundational work of Selberg and Roelcke \cite{Roelcke} on the eigenvalue problem for automorphic forms on the hyperbolic plane are treated systematically in \cite{Kubota}.   
For the case $\Gam=\Gam(1)$ one has Hecke original computation \cite{Hecke}. 
We do not reproduce this continuation argument here. Doing so correctly requires the general spectral theory of the weight $k$ hyperbolic Laplacian, in particular to rule out that the finite part above is nearly holomorphic. 
 I.e., that it is of the shape $f_1(\tau)+f_0(\tau)/(4\pi\Im\tau)$ for holomorphic $f_0$, $f_1$, 
 as for the Maass-Shimura completion of Remark \ref{rmkmaassshimura}, instead of holomorphic. This is outside the scope of the present paper. 
We do, however, sketch why the residue is forced to be $\tau$-independent and nonzero.
   As for the weight $0$ Eisenstein series, $E_{2,\mathfrak a}(\tau,s)$ has, at the cusp 
$\mathfrak a$, a Fourier expansion of the shape $y^s+\varphi(s)y^{1-s}+(\text{terms decaying as }y\to\infty)$, $y=\Im\tau$. Any pole in $s$ can therefore only arise from this constant term, and its residue is accordingly a single power of $y$ times a constant depending on neither $\Re\tau$ nor the decaying Fourier modes. The standard unfolding (Maass-Selberg) computation identifies this residue, up to a fixed combinatorial constant depending only on the weight, with a nonzero multiple of $1/\mathrm{vol}(\Gam\backslash\HH)$. This is the mechanism behind the classical fact $\mathrm{Res}_{s=1}E_{\mathfrak a}(\tau,s)=1/\mathrm{vol}(\Gam\backslash\HH)$ for the weight $0$ Eisenstein series (see, e.g., \cite{Iwaniec}). The weight $2$ residue at $s=0$ is regulated by the same mechanism. In particular $\kappa\ne0$ for every $\Gam$ of the first kind, since $0<\mathrm{vol}(\Gam\backslash\HH)<\infty$.
\end{proof}
%% 
%%%%%%%%%%%%%%%%%%%%%%%%%%%%%%%%%%%%%%%%%%%%%%%%%%%%%%%%%%%%%%%%%%%%%%%%%%%%%%%%%5
%% 
\begin{lemma} 
\label{lemfreeness}
Let $\Gam$ be any Fuchsian group and suppose $E\in Q_2(\Gam)$ is holomorphic on $\HH$, including at every cusp, with $Q_1(E)=c\ne0$. Then for every $k\ge 0$, and $J=\lfloor k/2\rfloor$, the multiplication map 
$$\mu:{\textstyle\bigoplus_{j=0}^{J}}M_{k-2j}(\Gam)\longrightarrow Q_k(\Gam),\qquad (h_j)_{0\le j\le J}\mapsto\sum_{j=0}^Jh_jE^j,$$ 
is a linear isomorphism. Equivalently, $Q(\Gam)=M(\Gam)[E]$ is a free polynomial ring of the single generator $E$ over the graded ring $M(\Gam)$.
\end{lemma}
%% 
%%%%%%%%%%%%%%%%%%%%%%%%%%%%%%%%%%%%%%%%%%%%%%%%%%%%%%%%%%%%%%%%%%%%%%%%%%%%%%%%%%%
%% 
\begin{proof}
Note that $h\in M_{k-2j}(\Gam)$ is holomorphic of depth $0$, thus $h|_{k-2j}\gamma=h$. 
Using $E|_2\gamma=E+cX(\gamma,\tau)$, we obtain 
$$(hE^j)|_k\gamma=h\cdot(E|_2\gamma)^j=h\sum_{m=0}^j\binom jm\big(cX(\gamma,\tau)\big)^mE^{j-m}=\sum_{m=0}^j\binom jmc^m\,X(\gamma,\tau)^m\big(hE^{j-m}\big),$$
 $hE^j\in Q_k(\Gam)_{\le j}$ with
\begin{equation}
\label{eqQmformula}
Q_m(hE^j)=\binom jmc^m\,hE^{j-m}\quad(0\le m\le j),\qquad Q_m(hE^j)=0,\quad m>j.
\end{equation}
By linearity, $\mu$ maps into $Q_k(\Gam)$, and for $f=\mu\big((h_j)_j\big)=\sum_jh_jE^j$,
\begin{equation}
\label{eqQmoff}
Q_m(f)=\sum_{j\ge m}\binom jmc^m\,h_jE^{j-m},\qquad 0\le m\le J.
\end{equation}

%%%%%%%%%%%%%%%%%%%%%%%%%%%%%%%%%%%%%%%%%%%%%%%%%%%%%%%%%%%%%%%%%%%%%%%%%%%%%%%%%%%%%%%%%%%%%
%% 
\emph{Injectivity.} We show that $h_J$, $h_{J-1}$, $\dots$, $h_0$ vanish in turn if $f=0$. 
Since $j\le J$, the sum in \eqref{eqQmoff} at $m=J$ has only the term $j=J$, thus $Q_J(f)=c^Jh_J$, as $f=0$, $h_J=0$. 
 Inductively, suppose $h_J=\cdots=h_{s+1}=0$ has been shown. Then in \eqref{eqQmoff} at $m=s$, only $j=s$ survives, the terms $j>s$ vanish as $h_j=0$, giving $Q_s(f)=c^sh_s$, as $f=0$, $h_s=0$. Downward induction from $s=J$ to $s=0$ gives $h_j=0$ for all $j$, thus $\ker\mu=0$.

%%%%%%%%%%%%%%%%%%%%%%%%%%%%%%%%%%%%%%%%%%%%%%%%%%%%%%%%%%%%%%%%%%%%%%%%%%%%%%%%%%%%%%%%%%5
%% 
\emph{Surjectivity.} Given $f\in Q_k(\Gam)$, we show that $f\in\mathrm{im}(\mu)$ by downward induction on $s$, where $f\in Q_k(\Gam)_{\le s}$, Definition \ref{defquasiscalar}. 
Every $f\in Q_k(\Gam)$ belongs to $Q_k(\Gam)_{\le J}$, thus the induction starts at $s=J$.   
 Let $h_s:=Q_s(f)\in M_{k-2s}(\Gam)$, which is modular by Proposition 
\ref{propdepthfiltration} for every $\Gam$.  
By \eqref{eqQmformula} with $j=m=s$, $Q_s(h_sE^s)=c^sh_s$, thus $f':=f-c^{-s}h_sE^s$ satisfies $Q_s(f')=Q_s(f)-c^{-s}Q_s(h_sE^s)=h_s-h_s=0$, i.e., $f'\in Q_k(\Gam)_{\le s-1}$.  
Repeating with $s-1$, $s-2$, $\dots$, $0$,   
at each stage subtracting $c^{-j}Q_j(\text{current remainder})\cdot E^j$, terminates after finitely many steps with remainder $0$, expressing $f=\sum_{j=0}^Jc^{-j}Q_j(\cdots)E^j\in\mathrm{im}(\mu)$. 
Since $\mu$ is a linear isomorphism for every $k$, multiplicative in $E$ 
and $M(\Gam)$-linear, $Q(\Gam)=\bigoplus_kQ_k(\Gam)$ is the free polynomial ring $M(\Gam)[E]$ of the single generator $E$. 
\end{proof}
%%
%%%%%%%%%%%%%%%%%%%%%%%%%%%%%%%%%%%%%%%%%%%%%%%%%%%%%%%%%%%%%%%%%%%%%%%%%%%%%%%%%%%%%%%%%%
Existence of $E_2^\Gam$ for groups with a cusp is given by 
\begin{proposition} 
\label{propE2exists}
Suppose $\Gam$ has at least one cusp ($n\ge 1$). Then there is a holomorphic function  
$E_2^\Gam:\HH\to\CC$ with a $q_\lambda$-expansion (no pole) at every cusp, of weight $2$ and depth $1$, normalized as in \eqref{eqE2anomaly} below.  
Moreover, $Q(\Gam)=\bigoplus_{m\ge 0}(E_2^\Gam)^mM(\Gam)$ is,  as for $\Gam(1)$, a free $M(\Gam)$-module on $1$, $E_2^\Gam$, $(E_2^\Gam)^2$, $\dots$, isomorphic as a graded ring to $M(\Gam)[E_2^\Gam]$
\begin{equation}
\label{eqE2anomaly}
E_2^\Gam|_2\gamma(\tau)=E_2^\Gam(\tau)+\frac{12}{2\pi i}X(\gamma,\tau), \qquad\gamma\in\Gam.
\end{equation}
\end{proposition}
%% 
%%%%%%%%%%%%%%%%%%%%%%%%%%%%%%%%%%%%%%%%%%%%%%%%%%%%%%%%%%%%%%%%%%%%%%%%%%%%%%%%%%%%%%%%
%% 
\begin{proof}
Apply Lemma \ref{lemheckefact} to a chosen cusp $\mathfrak a$ of $\Gam$, and rescale, 
$E_2^\Gam:=\big((12/2\pi i)/\kappa\big)E_{2,\mathfrak a}$.  
This satisfies \eqref{eqE2anomaly} and inherits holomorphy on $\HH$ and at every cusp of $\Gam$, not only at $\mathfrak a$, from Lemma \ref{lemheckefact}, since the finite part 
$E_{2,\mathfrak a}$ was already shown there to have a holomorphic $q_\lambda$-expansion at every cusp.  
 This proves existence with the stated transformation law, and in particular expresses 
$E_2^\Gam\in Q_2(\Gam)$ with $Q_0(E_2^\Gam)=E_2^\Gam$, $Q_1(E_2^\Gam)=12/2\pi i\ne 0$. By Definition \ref{defquasiscalar} this is depth $\le 1$, and $Q_1\ne 0$ excludes depth $0$, thus the depth is $1$.  
Freeness of $Q(\Gam)=M(\Gam)[E_2^\Gam]$ is given by Lemma \ref{lemfreeness} applied with $E=E_2^\Gam$, $c=12/2\pi i$, which uses nothing about $E_2^\Gam$ beyond the two facts just established (holomorphy, and $Q_1=12/2\pi i\ne0$) and no property of $\Gam(1)$. 
\end{proof}
%% 
%%%%%%%%%%%%%%%%%%%%%%%%%%%%%%%%%%%%%%%%%%%%%%%%%%%%%%%%%%%%%%%%%%%%%%%%%%%%%%%%%%%%%%%%%%
%% 
 Note that freeness of $Q(\Gam)=M(\Gam)[E_2^\Gam]$, the fact on which Theorems \ref{thmdoubledQVOA}, \ref{thmkernel}, \ref{thmsurjectivity}, \ref{thmPop}, and Corollary \ref{cordimformula} depend, is given by Lemma \ref{lemfreeness}. 
%%
%%%%%%%%%%%%%%%%%%%%%%%%%%%%%%%%%%%%%%%%%%%%%%%%%%%%%%%%%%%%%%%%%%%%%%%%%%%%%%%%%%%%%5
%% 
\begin{remark}
\label{rmkwhyfails} 
It is worth mentioning  why the logarithmic-derivative approach, i.e., 
  constructing $E_2^\Gam$ as $\partial\log F$ for a suitable automorphic form $F$ attached to the cusp cannot succeed by any choice of $F$ once $g\ge 1$.  
A nonvanishing on $\HH$ holomorphic form $F\in M_w(\Gam)$, $w>0$, 
 the input to avoid the logarithmic derivative inheriting poles at the zeros of $F$, 
 corresponds to a section of the degree-$w(g-1)$ line bundle $L_w$ on $X$ (Proposition \ref{propRR}) whose divisor is supported  entirely at the cusps. 
 For $g=0$ this is automatic, i.e., line bundles on $\mathbb P^1$ are classified by degree alone, thus any effective divisor of the right degree supported at the cusps already represents $L_w$, and $\Delta$ is this for $\Gam(1)$, $w=12$.  
 For $g\ge 1$, by contrast, $\mathrm{Pic}^d(X)$ is the $g$-dimensional Jacobian variety of $X$ (see, e.g., \cite{Gunning}), and a divisor supported at the finitely many cusps represents only a thin, in general empty for a generic Fuchsian group, subvariety of $\mathrm{Pic}^d(X)$.   
The direct construction above, constructing $E_2^\Gam$ from a convergent, after regularization, series instead of as the logarithmic derivative of an auxiliary form,
 bypasses this nonexistence issue entirely. 
\end{remark}
%%
%%%%%%%%%%%%%%%%%%%%%%%%%%%%%%%%%%%%%%%%%%%%%%%%%%%%%%%%%%%%%%%%%%%%%%%%%%%%%%%%%%%%%%%%
%% 
We isolate, for use below, the part of the depth-filtration theory that survives  in any case, with no hypothesis on $\Gam$, since it never produces a holomorphic connection and so is never obstructed by Theorem \ref{thmatiyah}.
%% 
%%%%%%%%%%%%%%%%%%%%%%%%%%%%%%%%%%%%%%%%%%%%%%%%%%%%%%%%%%%%%%%%%%%%%%%%%%%%%%%%%%%%
%% 
Classical unobstructed part of the depth filtration is described in 
\begin{proposition}
\label{propdepthfiltration}
For every Fuchsian group $\Gam$, the natural maps $f\mapsto Q_s(f)$ identify the graded pieces of the depth filtration on $Q_k(\Gam)$ with subspaces 
$$Q_k(\Gam)_{\le s}/Q_k(\Gam)_{\le s-1}\;\hookrightarrow\;M_{k-2s}(\Gam), 
\qquad 0\le s\le k/2,$$ 
so that $\dim Q_k(\Gam)\le\sum_{i=0}^{\lfloor k/2\rfloor}\dim M_{k-2i}(\Gam)$, with equality for every $k$ if $\Gam$ has a cusp (Proposition \ref{propE2exists}). 
 For $\Gam$ cocompact of genus $g\ge 2$, equality already fails at $k=2$ by Theorem \ref{thmatiyah}: $\dim Q_2(\Gam)=\dim M_2(\Gam)=g<g+1=\dim M_0(\Gam)+\dim M_2(\Gam)$.
\end{proposition}
%% 
%%%%%%%%%%%%%%%%%%%%%%%%%%%%%%%%%%%%%%%%%%%%%%%%%%%%%%%%%%%%%%%%%%%%%%%%%%%%%%%%%%
%% 
\begin{proof}
Injectivity of the graded pieces (the embedding direction) is the elementary half of the classical theory of nearly holomorphic forms  
 \cite{ShimuraNH} (see also \cite{Zagier}).  
 It only uses that $Q_s$ is well defined and additive with no surjectivity input.
 The equality statement for groups with a cusp is given by Proposition \ref{propE2exists}, and the failure of equality at $k=2$ for cocompact groups is given by Theorem \ref{thmatiyah}. 
\end{proof}
%%
%%%%%%%%%%%%%%%%%%%%%%%%%%%%%%%%%%%%%%%%%%%%%%%%%%%%%%%%%%%%%%%%%%%%%%%%%%%%%%%%%%%%
%% 
\begin{remark} 
\label{rmkmaassshimura}
Although a literal holomorphic generator $E_2^\Gam$ does not exist for cocompact $\Gam$, the Maass-Shimura raising operator 
$$\delta_kf:=\frac1{2\pi i}\frac{\partial f}{\partial\tau}+\frac{k}{4\pi\Im\tau}f,$$ 
is well defined for every $\Gam$.  
It depends only on the hyperbolic metric coefficient $\Im\tau$, via the elementary   identity $\Im(\gamma\tau)=\Im(\tau)/|j(\gamma,\tau)|^2$, $\gamma\in\PSL_2(\RR)$. 
It sends a weight $k$  holomorphic automorphic form $f$ to a real-analytic, weight 
$(k+2)$-automorphic, nearly holomorphic form $\delta_kf=f_1(\tau)+f_0(\tau)/(4\pi\Im\tau)$  wher $f_0$, $f_1$ are holomorphic.   
Geometrically, $\delta_k$ is the Chern connection of the always existing, smooth hyperbolic Hermitian metric on $\omega^{\otimes k/2}$. 
 It is automatically compatible with the holomorphic structure in the $\bar\partial$-direction, and Atiyah obstruction. 
 We use $\delta_k$ in Section \ref{seccocompact} to obtain partial, unobstructed results in the cocompact case.
\end{remark}
%%
%%%%%%%%%%%%%%%%%%%%%%%%%%%%%%%%%%%%%%%%%%%%%%%%%%%%%%%%%%%%%%%%%%%%%%%%%%%%%%%%%
%% 
\section{The vertex operator algebra bundle on $\Gam\backslash\HH$}\label{secbundle}
We now recall the construction of \cite{BCFMN} and observe that it goes through unchanged for an arbitrary Fuchsian group. 

Let $\Aut\mathcal O$ be the group of formal coordinate changes, $\HH\times\Aut\mathcal O\to\HH$ the torsor of local coordinates, and for a QVOA $V$ let $R(f)$ denote the action of $f\in\Aut\mathcal O$ on $V$ via the Virasoro operators $L(n)$, $n\ge 0$ \cite{BCFMN}. 
 Form the fiber product  
$V_\HH=(\HH\times\Aut\mathcal O)\times_{\Aut\mathcal O}V$ and use Lemma 3.4 of loc. cit. 
(an identity purely about $\Aut\mathcal O$, with no reference to any discrete group) 
to obtain the  
description of the resulting left action of $\GL_2(\RR)^+$ on $\HH\times V$ 
\begin{equation}
\label{eqKaction}
\gamma(\tau,v)=\left(\gamma\tau,\;e^{-\det(\gamma)^{-1}cj(\gamma,\tau)L(1)}j(\gamma,\tau)^{-2L(0)}(\det\gamma)^{L(0)}v\right)=(\gamma\tau,K(\gamma,\tau)v),
\end{equation}
where $K(\gamma,\tau)=e^{-\det(\gamma)^{-1}cj(\gamma,\tau)L(1)}j(\gamma,\tau)^{-2L(0)}(\det\gamma)^{L(0)}\in\End(V)$, as in \cite{BCFMN}.
%% 

%%%%%%%%%%%%%%%%%%%%%%%%%%%%%%%%%%%%%%%%%%%%%%%%%%%%%%%%%%%%%%%%%%%%%%%%%%%%%%%%%%%%%%%%%
%% 
Generalizing result of \cite{BCFMN} we obtain 

\begin{theorem}
\label{thmbundle}
For every $\gamma$, $\beta\in\GL_2(\RR)^+$ one has the cocycle identity 
$K(\gamma\beta,\tau)=K(\gamma,\beta\tau)K(\beta,\tau)$, thus \eqref{eqKaction}
 defines a left action of $\GL_2(\RR)^+$ on $\HH\times V$.  
In particular, for every Fuchsian group $\Gam\subset\PSL_2(\RR)$ (lifted to $\SL_2(\RR)$, or worked with inside $\GL_2(\RR)^+$ as in \cite{BCFMN}) and every QVOA $V$, the quotient 
$$\mathcal V_Y:=\Gam\backslash(\HH\times V),$$
is a well defined (orbifold, in the sense of Section \ref{secfuchsian}, ordinary once $r=0$) vector bundle on $Y=\Gam\backslash\HH$, extending canonically to a bundle  
$\mathcal V_X$ on the compactified curve $X$  as in \cite{BCFMN}.
\end{theorem}
%% 
%%%%%%%%%%%%%%%%%%%%%%%%%%%%%%%%%%%%%%%%%%%%%%%%%%%%%%%%%%%%%%%%%%%%%%%%%%%%%%%%%%%%%%%
%% 
\begin{proof}
The proof that \eqref{eqKaction} is a group action, both the derivation via Lemma 3.4 of 
\cite{BCFMN} and the direct verification in \cite{BCFMN} using Lemma 3.5 of 
 loc. cit., $X(\alpha)|_2\beta=X(\alpha\beta)-X(\beta)$, is an identity among matrices 
$\alpha$, $\beta\in\GL_2(\RR)^+$ and the operators $L(-1)$, $L(0)$, $L(1)$ satisfying the 
$\mathfrak{sl}_2$ relations. Note that at no point is $\Gam(1)$, or any property special to it, used.
 The same is true of Lemma 3.7 of loc. cit., 
which is logically equivalent to the group-action property.   
Restricting the resulting action of $\GL_2(\RR)^+$ on $\HH\times V$ to an arbitrary subgroup $\Gam$ therefore gives a well defined action. The quotient construction, its extension across the cusps via the square-bracket/$q$-expansion formalism of \cite{BCFMN} (again an identity in $\Aut\mathcal O$, valid for any cusp of any Fuchsian group) proceed without change.
\end{proof}
%%
%%%%%%%%%%%%%%%%%%%%%%%%%%%%%%%%%%%%%%%%%%%%%%%%%%%%%%%%%%%%%%%%%%%%%%%%%%%%%%%%%%%%%%%%%
%% 
\begin{definition}
\label{defVvaluedforms}
A \emph{weakly holomorphic $V$-valued automorphic form of weight $k$} for $\Gam$ of conformal degree $\le N$ is a holomorphic map $f:\HH\to F_NV$ with $f\Vert_k\gamma=f$ for all 
$\gamma\in\Gam$ where
$$
f\Vert_k\gamma(\tau):=j(\gamma,\tau)^{-k}(\det\gamma)^{k/2}K(\gamma,\tau)^{-1}f(\gamma\tau). $$ 
We say $f$ is holomorphic if, in addition,  

\noindent\quad(a) (cusps) at every cusp of $\Gam$, $f$ has a holomorphic $q_\lambda$-expansion $f\in F_NV[[q_\lambda]]$;

\noindent\quad(b) (elliptic points) at every elliptic point $\tau_i$ with stabilizer generator $\gamma_i$ of order $m_i$, the vector $f(\tau_i)\in F_NV$ is fixed by the finite-order operator $j(\gamma_i,\tau_i)^kK(\gamma_i,\tau_i)$ on $F_NV$.

Let $M_k(V,\Gam)$ denote the resulting space, $M(V,\Gam)=\bigoplus_kM_k(V,\Gam)$ with its conformal filtration $F_NM(V,\Gam)$  as in \cite{BCFMN}.
\end{definition}
%% 
%%%%%%%%%%%%%%%%%%%%%%%%%%%%%%%%%%%%%%%%%%%%%%%%%%%%%%%%%%%%%%%%%%%%%%%%%%%%%%%%%%%% 
%% 
Condition (b) has no counterpart in \cite{BCFMN} because the only nontrivial-looking elliptic points of $\Gam(1)$, i.e., at $\tau=i$ and $\tau=\rho=e^{i\pi/3}$,  
 happen to act trivially on $V$ in the relevant sense, \cite{BCFMN} remarks that  
$\langle-1,T\rangle$ acts trivially on $V$.
 For a general $\Gam$ with nontrivial elliptic points, $K(\gamma_i,\tau_i)$ involves the nilpotent operator $L(1)$ as well as the semisimple $j(\gamma_i,\tau_i)^{-2L(0)}$, and condition (b) cuts out a linear subspace of $F_NV$ at each elliptic point, strictly refining the classical scalar condition \eqref{eqellipticscalar}.  
%% 
%%%%%%%%%%%%%%%%%%%%%%%%%%%%%%%%%%%%%%%%%%%%%%%%%%%%%%%%%%%%%%%%%%%%%%%%%%%%%%%%%%%%
%% 
Finite order at elliptic points is described in 

\begin{lemma}
\label{lemellipticfiniteorder}
Let $\gamma_i\in\Gam$ generate the order-$m_i$ stabilizer of the elliptic point $\tau_i$. Then $K(\gamma_i,\tau_i)^{m_i}=\Id_{F_NV}$ for every $N$. Consequently $j(\gamma_i,\tau_i)^kK(\gamma_i,\tau_i)$, restricted to $F_NV$, is an operator of finite order, and hence is diagonalizable over $\CC$.
\end{lemma}
%% 
%%%%%%%%%%%%%%%%%%%%%%%%%%%%%%%%%%%%%%%%%%%%%%%%%%%%%%%%%%%%%%%%%%%%%%%%%%%%%%%%%%%%%
%% 
\begin{proof}
Since $\gamma_i\tau_i=\tau_i$, the cocycle identity $K(\alpha\beta,\tau)=K(\alpha,\beta\tau)K(\beta,\tau)$ of Theorem \ref{thmbundle} gives, by induction on $r$, $K(\gamma_i^r,\tau_i)=K(\gamma_i,\tau_i)^r$ for every $r\ge 1$.   
Lifting $\gamma_i$ to $\GL_2(\RR)^+$, the hypothesis that $\gamma_i$ has order $m_i$ in 
$\PSL_2(\RR)$ means $\gamma_i^{m_i}=\epsilon\Id$ for $\epsilon=\pm1$.  
  From the formula for $K$ in Definition \ref{defVvaluedforms},  
$K(\epsilon\Id,\tau)=e^0\cdot\epsilon^{-2L(0)}\cdot1^{L(0)}=\Id_V$ for either sign of 
$\epsilon$, as $\epsilon^{-2}=1$,  
thus $K(\gamma_i,\tau_i)^{m_i}=K(\gamma_i^{m_i},\tau_i)=\Id_V$ on every $F_NV$.   
Since $j(\gamma_i,\tau_i)$ is a root of unity of order dividing $2m_i$, the classical multiplier of $\gamma_i$ at its fixed point cf. \eqref{eqellipticscalar} and \cite{Shimura,  Miyake}, the operator $j(\gamma_i,\tau_i)^kK(\gamma_i,\tau_i)$ also satisfies a separable polynomial $x^{2m_i}-1=0$ over $\CC$, and a finite-dimensional operator annihilated by a separable polynomial is diagonalizable.
\end{proof}
%% 
%%%%%%%%%%%%%%%%%%%%%%%%%%%%%%%%%%%%%%%%%%%%%%%%%%%%%%%%%%%%%%%%%%%%%%%%%%%%%%%%%%%%%%%%%%
%% 
\begin{remark}
\label{rmknotjordan}
 Let us explain why the above does not force $f(\tau_i)\in\ker L(1)$.  
It is tempting, but, as we now explain, not generally correct, to argue as follows. 
 $K(\gamma_i,\tau_i)=e^{-cj(\gamma_i,\tau_i)L(1)}\cdot j(\gamma_i,\tau_i)^{-2L(0)}(\det\gamma_i)^{L(0)}$ expresses $K$ as a product of a unipotent factor $U=e^{-cjL(1)}$  
as $L(1)$ is nilpotent on each $F_NV$, and a semisimple factor $S=j^{-2L(0)}(\det)^{(0)}$.
Since $K$ has finite order according to Lemma \ref{lemellipticfiniteorder}, one might expect this forces $U=\Id$, i.e., $f(\tau_i)\in\ker L(1)$.  
 This reasoning would be correct if $U$ and $S$ commuted, making $US$ the multiplicative Jordan-Chevalley decomposition of $K$ into commuting semisimple and unipotent parts, for which finite order of the product does force the unipotent factor to be trivial. 
 But $L(0)$ and $L(1)$ do not commute, i.e., $[L(0),L(1)]=-L(1)$), 
thus $U$ and $S$ do not commute, $US$ is not the Jordan-Chevalley decomposition of $K$, and finite order of $K$ does not force $U=\Id$. Proposition \ref{propellipticstructure} below gives the correct structural statement. 
\end{remark}
%% 
%%%%%%%%%%%%%%%%%%%%%%%%%%%%%%%%%%%%%%%%%%%%%%%%%%%%%%%%%%%%%%%%%%%%%%%%%%%%%%%%%%%%%%%%%%%
%% 
The Proposition below describes the elliptic-point eigenspace, generic case. 

\begin{proposition} 
\label{propellipticstructure}
With notation as in Lemma \ref{lemellipticfiniteorder}, fix $N$ and set 
 $\lambda_n:=j(\gamma_i,\tau_i)^{k-2n}$ for $0\le n\le N$, the eigenvalue by which 
$j(\gamma_i,\tau_i)^kK(\gamma_i,\tau_i)$ acts on the associated graded piece 
$F_nV/F_{n-1}V\cong V_n$.    
Suppose $\lambda_{n_0}=1$ for a unique $n_0\in\{0,\dots,N\}$ with $V_{n_0}\ne 0$,  
and $\lambda_n\ne 1$ for every other such $n$. 
 Then the fixed subspace $\{v\in F_NV:j(\gamma_i,\tau_i)^kK(\gamma_i,\tau_i)v=v\}$ has dimension $\dim V_{n_0}$, and projection onto the $V_{n_0}$-component is an isomorphism from this fixed subspace onto $V_{n_0}$.
 Every $v_0\in V_{n_0}$ extends uniquely to a fixed vector $v_0+v_0'$, $v_0'\in F_{n_0-1}V$,  with no further constraint on $v_0$ beyond the resonance condition $\lambda_{n_0}=1$ already assumed.  
If no $n\le N$ with $V_n\ne0$ has $\lambda_n=1$, the fixed subspace is $0$. 
\end{proposition}
%% 
%%%%%%%%%%%%%%%%%%%%%%%%%%%%%%%%%%%%%%%%%%%%%%%%%%%%%%%%%%%%%%%%%%%%%%%%%%%%%%%%%%%%%%%%
%% 
\begin{proof}
By Lemma \ref{lemellipticfiniteorder}, $\hat K:=j(\gamma_i,\tau_i)^kK(\gamma_i,\tau_i)$ is diagonalizable on $F_NV$. Since $L(1)$ strictly lowers conformal degree, $\hat K$ preserves the filtration $F_0V\subset\cdots\subset F_NV$ and induces the scalar $\lambda_n$ on $F_nV/F_{n-1}V$. 
 Write an expected fixed vector as $v=v_0+v_0'$, $v_0\in V_{n_0}$, $v_0'\in F_{n_0-1}V$, and expand the equation $\hat Kv=v$ degree by degree from the top. 
Then the degree-$n_0$ equation is $\lambda_{n_0}v_0=v_0$, automatic by hypothesis.  
For each $n<n_0$, the degree-$n$ equation has the form 
$\lambda_nv_n'+L_n(v_0,\dots)=v_n'$ for a linear expression $L_n$ built from $v_0$ via iterated application of $L(1)$ and, already-determined by downward induction, higher degree components of $v_0'$. Hence, it is uniquely solvable for the degree-$n$ component $v_n'$ 
because $\lambda_n\ne 1$ makes $\lambda_n-1$ invertible.  
 This produces, for every $v_0\in V_{n_0}$, a unique completion to a fixed vector, with no constraint imposed on $v_0$, proving the claim.
 The dimension count and the final claim (no resonant degree) follow immediately, again using diagonalizability to rule out any further generalized-eigenvector contribution.
\end{proof}
%%
%%%%%%%%%%%%%%%%%%%%%%%%%%%%%%%%%%%%%%%%%%%%%%%%%%%%%%%%%%%%%%%%%%%%%%%%%%%%%%%%%%%%%5
%% 
\begin{remark}
\label{rmkellipticgeneric}
Proposition \ref{propellipticstructure} shows that condition (b) does not, in general, force $f(\tau_i)$ into $\ker L(1)$. In the generic case of a unique resonant degree $n_0$ within $F_NV$, the leading term $v_0\in V_{n_0}$ is unconstrained beyond the resonance condition.   It is completed canonically by lower-degree corrections built from 
$L(1)v_0$, $L(1)^2v_0$, $\dots$,  the unobstructed completion already familiar from constant sections (cf. \cite{BCFMN}).
 There is, however, a basic and important special case: when $f$ is  a constant section valued in a single graded piece $V_n$ (as in \cite{BCFMN} and our constructions,  e.g., Proposition \ref{proptopexact} below), necessarily $N=n_0=n$, thus there is no room below $n_0$ for any completion $v_0'\in F_{n_0-1}V$.  
 In that case the resonance condition $\lambda_{n_0}=1$ becomes a constraint on $f$, with no lower-degree freedom to absorb its failure, and quasi-primarity is not only sufficient but necessary. We mention this in Example 
\ref{exelliptic}. The delicate case, i.e., multiple resonant degrees within a single $F_NV$, is not needed below. 
Diagonalizability of $\hat K$, Lemma \ref{lemellipticfiniteorder}, still applies there, but forces compatibility conditions linking the $L(1)$-orbits of the resonant pieces that
 we do not consider in this paper.

It is needed to state, whether a generic-resonance hypothesis is in force. Definition \ref{defVvaluedforms}(b)  requires no such hypothesis. Namely, 
 $f(\tau_i)$ is fixed by $j(\gamma_i,\tau_i)^kK(\gamma_i,\tau_i)$ is simply the condition that a vector lie in the fixed subspace of a specific linear operator, which is well defined on every $F_NV$ regardless of how many resonant degrees $\hat K$ has.
 Proposition \ref{propellipticstructure} is not a hypothesis on $f$ but an attempt to 
 describe this fixed subspace, and it succeeds only in the generic case. 
 Consequently, Theorem \ref{thmbundle}, Definition \ref{defVvaluedforms}, and Theorem \ref{thmconnection} use no property of the fixed subspace beyond its being well defined, hence require no resonance hypothesis of any kind.  
 Proposition \ref{proptopexact}, and, via it, Theorem \ref{thmexactdim} and Proposition \ref{proppartialconj}, also require none.
 Their proofs work with constant sections $f\equiv v\in V_n$, the extreme case $N=n_0=n$ of Proposition \ref{propellipticstructure} in which there is only one graded piece and hence only one candidate resonant degree (cf. Example \ref{exelliptic}). 
Thus the generic case is automatic there and Proposition \ref{propellipticstructure} is not even  used. 
 The only place in this paper where the generic-resonance hypothesis is used  
is the general-signature case of Conjecture \ref{conjexactdim} beyond its top-degree term,  as mentioned in Remark \ref{rmktowardsconj}. A complete proof of Conjecture \ref{conjexactdim} would need to consider the multiple-resonance case, which we have not done.
\end{remark}
%%
%%%%%%%%%%%%%%%%%%%%%%%%%%%%%%%%%%%%%%%%%%%%%%%%%%%%%%%%%%%%%%%%%%%%%%%%%%%%%%%%5
%% 
\section{The connection on $\mathcal V_X$}\label{secconnection}
The following result generalizes \cite{BCFMN}. 

\begin{theorem} 
\label{thmconnection}
$$\nabla:=d+L(-1)\otimes d\tau,$$
defines a connection $\nabla:\mathcal V_X\to\mathcal V_X\otimes\Omega^1_X$ on the bundle $\mathcal V_X$ for every Fuchsian group $\Gam$.
\end{theorem}
%% 
%%%%%%%%%%%%%%%%%%%%%%%%%%%%%%%%%%%%%%%%%%%%%%%%%%%%%%%%%%%%%%%%%%%%%%%%%%%%%%%%%%%%%%%%
%% 
\begin{proof}
The proof of \cite{BCFMN} establishes, for $f$ a $V$-valued holomorphic function satisfying 
 $f(\gamma\tau)=$ $e^{-cj(\gamma,\tau)L(1)}$ 
 $\cdot j(\gamma,\tau)^{k-2L(0)}f(\tau)$ for an arbitrary 
 $\gamma=\left(\begin{smallmatrix}a&b\\c&d\end{smallmatrix}\right)\in\GL_2(\RR)^+$, the identity
\begin{equation}
\label{eqnablak}
(\nabla_kf)(\gamma\tau)=e^{-cj(\gamma,\tau)L(1)}j(\gamma,\tau)^{k+2-2L(0)}(\nabla_kf)(\tau),\qquad \nabla_kf:=\nabla f-\frac{2\pi ik}{12}E_2f,
\end{equation}
where the proof reduces, after substituting the transformation law of $E_2$ (their eq. (7)) to cancel the homogeneous-weight part of the anomaly, to the purely Lie-algebraic identity (their eq.\ (9))
$$L(-1)f_N=j(\gamma,\tau)^{2L(0)-2N-2}e^{cj(\gamma,\tau)L(1)}L(-1)e^{-cj(\gamma,\tau)L(1)}f_N-c^2j(\gamma,\tau)^{-2}L(1)f_N-2Ncj(\gamma,\tau)^{-1}f_N,$$  
$f_N\in V_N$, verified there using only the Virasoro commutation relations $[L(1),L(-1)]=2L(0)$ 
 and $[L(m),L(n)]=(m-n)L(m+n)$. 
 This identity involves only $L(-1)$, $L(0)$, $L(1)$, the scalars $c$, $j(\gamma,\tau)$, and the grading $N$. 
It makes no reference to $\Gam(1)$ and holds for every $\gamma\in\GL_2(\RR)^+$, 
 hence for every $\gamma$ in every Fuchsian group $\Gam$.
Specializing \eqref{eqnablak} to $k=0$ (so that the term $\tfrac{2\pi ik}{12}E_2f$ vanishes identically, and the existence question for $E_2^\Gam$ is irrelevant) gives  the statement that $\nabla=\nabla_0=d+L(-1)\otimes d\tau$ transforms $V$-valued functions on $\HH$ into $(\mathcal V_X\otimes\Omega^1_X)$-valued ones in the manner required for a connection.  
Since the computation is valid for every $\gamma\in\Gam$, $\nabla$ descends to the connection  claimed. 
\end{proof}
%% 
%%%%%%%%%%%%%%%%%%%%%%%%%%%%%%%%%%%%%%%%%%%%%%%%%%%%%%%%%%%%%%%%%%%%%%%%%%%%%%%%%%%%%%
%% 
For weight $k\ne0$, when $\Gam$ has at least one cusp one may take the modular derivative 
$$\nabla_kf:=\nabla f-\frac{2\pi ik}{12}E_2^\Gam f,$$ 
with the generator $E_2^\Gam$ of Proposition \ref{propE2exists}. 
 Equation \eqref{eqnablak}, with $E_2$ replaced by $E_2^\Gam$ using only that \eqref{eqE2anomaly} holds, then  shows that $\nabla_k: M_k(V,\Gam)\to M_{k+2}(V,\Gam)$,  as in \cite{BCFMN}. 
 When $\Gam$ is cocompact of genus $g\ge 2$, Theorem \ref{thmatiyah} shows that no correction term built algebraically from $Q(\Gam)$ can make this modular derivative preserve automorphy for $k\ne 0$.
 The connection $\nabla$ of Theorem \ref{thmconnection} remains available and unobstructed. 
 But the refined, weight-preserving modular derivative $\nabla_k$ is a phenomenon special to groups with a cusp. 
%%
%%%%%%%%%%%%%%%%%%%%%%%%%%%%%%%%%%%%%%%%%%%%%%%%%%%%%%%%%%%%%%%%%%%%%%%%%%%%%%%%%%%%%%%
%% 
\section{Quasi-automorphic $V$-valued forms and the doubled QVOA $Q(V,\Gam)$}
\label{secquasi}    
\begin{definition}
\label{defVvalQuasi}
A \emph{$V$-valued meromorphic quasi-automorphic form of weight $k$ and depth $\le s$ for  
$\Gam$} is a holomorphic map $f:\HH\to F_NV$ for some $N$ such that
$$f\Vert_k\gamma(\tau)=\sum_{m=0}^sX(\gamma,\tau)^mQ_m(f)(\tau),$$ 
for holomorphic $Q_m(f):\HH\to F_NV$, with $f$ meromorphic at the cusps. 
 Let $Q_k(V,\Gam)$ be the resulting space and $Q(V,\Gam)=\bigoplus_{k\ge 0}Q_k(V,\Gam)$.
\end{definition}
%%
%%%%%%%%%%%%%%%%%%%%%%%%%%%%%%%%%%%%%%%%%%%%%%%%%%%%%%%%%%%%%%%%%%%%%%%%%%%%%%%%%%%%
%% 
\begin{theorem} 
\label{thmQiso}
There is a canonical isomorphism of graded $Q(\Gam)$-modules 
$$\iota: Q(\Gam)\otimes V^{(2)}\xrightarrow{\sim}Q(V,\Gam),
\qquad f(\tau)\otimes v\mapsto f(\tau)v.$$
\end{theorem}
%% 
%%%%%%%%%%%%%%%%%%%%%%%%%%%%%%%%%%%%%%%%%%%%%%%%%%%%%%%%%%%%%%%%%%%%%%%%%%%%%%%%%%%%%
%% 
\begin{proof}
 Injectivity and the fact that $\iota$ is a graded $Q(\Gam)$-module map are immediate computations using only Definition \ref{defVvalQuasi}  which makes no reference to $\Gam(1)$. 
Surjectivity is proved by the inductive extraction of the top conformal-degree coefficient, 
 an argument depending only on the structure of $\Aut\mathcal O$-modules of finite conformal degree, again with no use of $\Gam(1)$.
\end{proof}
%% 
%%%%%%%%%%%%%%%%%%%%%%%%%%%%%%%%%%%%%%%%%%%%%%%%%%%%%%%%%%%%%%%%%%%%%%%%%%%%%%%%%%%%5
%% 
Next we derive the Ramanujan-type identity for $\theta E_2^\Gam$  from the depth-filtration recursion of Lemma \ref{lemthetadepth} below, instead of by a dimension count or by analogy with $\Gam(1)$. 
%% 
%%%%%%%%%%%%%%%%%%%%%%%%%%%%%%%%%%%%%%%%%%%%%%%%%%%%%%%%%%%%%%%%%%%%%%%%%%%%%%%%%%%%%%%55
%% 
\begin{lemma} 
\label{lemthetadepth}
Let $g\in Q_k(\Gam)$ have depth $\le p$, thus  
$g|_k\gamma=\sum_{m=0}^pX(\gamma,\tau)^mQ_m(g)(\tau)$. 
Then $\theta g\in Q_{k+2}(\Gam)$ has depth $\le p+1$ with
\begin{equation}
\label{eqthetarecursion}
Q_m(\theta g)=\theta\big(Q_m(g)\big)+\frac{k-m+1}{2\pi i}\,Q_{m-1}(g), \qquad 0\le m\le p+1, 
\end{equation}
using the conventions $Q_{-1}(g):=0$, $Q_{p+1}(g):=0$.
\end{lemma}
%% 
%%%%%%%%%%%%%%%%%%%%%%%%%%%%%%%%%%%%%%%%%%%%%%%%%%%%%%%%%%%%%%%%%%%%%%%%%%%%%%%%%%%%%%
%% 
\begin{proof}
Write $g(\gamma\tau)=j(\gamma,\tau)^k\sum_{m=0}^pX(\gamma,\tau)^mQ_m(g)(\tau)$ and differentiate both sides in $\tau$ for $\gamma=\left(\begin{smallmatrix}a&b\\c&d\end{smallmatrix}\right)\in\Gam\subset\SL_2(\RR)$, thus $j=c\tau+d$, $j'=c$, and $\tfrac{d}{d\tau}(\gamma\tau)=j^{-2}$ as $\det\gamma=1$.
 The left side differentiates to $g'(\gamma\tau)j^{-2}$.
 For the right side, $X=c/j$ gives $X'=-c\,j'/j^2=-c^2/j^2=-X^2$, and $j^k{}'=kcj^{k-1}=kXj^k$, using $c=Xj$.
 Thus
\begin{eqnarray*}
\frac{d}{d\tau}\big[j^kX^mQ_m(g)\big]
&=&kXj^kX^mQ_m(g)-mj^kX^{m+1}Q_m(g)+j^kX^mQ_m(g)'
\\
&=&j^k\big[(k-m)X^{m+1}Q_m(g)+X^mQ_m(g)'\big].
\end{eqnarray*}
Summing over $m=0$, $\dots$, $p$ equating with 
$g'(\gamma\tau)j^{-2}=j^{k+2}\cdot(\text{RHS})/j^k$, 
and reindexing the first term ($m\mapsto m-1$) gives
$$g'(\gamma\tau)=j^{k+2}\sum_{m=0}^{p+1}X^m\Big[Q_m(g)'+(k-m+1)Q_{m-1}(g)\Big],$$
i.e., $(g')|_{k+2}\gamma=\sum_mX^m\big[Q_m(g)'+(k-m+1)Q_{m-1}(g)\big]$.  
As $g'=2\pi i\,\theta g$ and $Q_m$ is linear, dividing by $2\pi i$ gives  
\eqref{eqthetarecursion}. 
\end{proof}
%%
%%%%%%%%%%%%%%%%%%%%%%%%%%%%%%%%%%%%%%%%%%%%%%%%%%%%%%%%%%%%%%%%%%%%%%%%%%%%%%%%%%%%%%%%%%55
%% 
Next we describe the higher genus Ramanujan identity and Serre derivative. 

\begin{lemma}
\label{lemramanujan}
Suppose $\Gam$ has a cusp, with $E:=E_2^\Gam$ as in Proposition \ref{propE2exists} 
thus $Q_0(E)=E$, $Q_1(E)=c:=12/2\pi i$.   
Then 
%%  
%%%%%%%%%%%%%%%%%%%%%%%%%%%%%%%%%%%%%%%%%%%%%%%%%%%%%%%%%%%%%%%%%%%%%%%%%%%%%
%% 
\begin{enumerate}[label=(\roman*)]
\item $E_4^\Gam:=E^2-12\,\theta E\ \in\ M_4(\Gam)$, and consequently
$$\theta E_2^\Gam=\frac{1}{12}\left((E_2^\Gam)^2-E_4^\Gam\right);$$
\item for every $h\in M_{k'}(\Gam)$, 
$\vartheta(h):=\theta h-\tfrac{k'}{12}E_2^\Gam h\ \in\ M_{k'+2}(\Gam)$, the classical Serre derivative, i.e., 
$$\theta h=\frac{k'}{12}E_2^\Gam h+\vartheta(h).$$
\end{enumerate}
\end{lemma}
%%
%%%%%%%%%%%%%%%%%%%%%%%%%%%%%%%%%%%%%%%%%%%%%%%%%%%%%%%%%%%%%%%%%%%%%%%%%%%%%%%%%%%%%%%%%
%%  
\begin{proof}
(ii) Apply Lemma \ref{lemthetadepth} to $g=h$ for $k=k'$, $p=0$, $Q_0(h)=h$, and we obtain  $Q_1(\theta h)=\theta(Q_1(h))+\tfrac{k'}{2\pi i}Q_0(h)=0+\tfrac{k'}{2\pi i}h$,  
since $Q_1(h)=0$, and $h$ has depth $0$.  
 Thus $\theta h\in Q_{k'+2}(\Gam)_{\le 1}$. By Lemma \ref{lemfreeness}, with $j=1$ applied to $Q_{k'+2}(\Gam)$), an element of depth $\le 1$ is uniquely $h'E+h''$ with $h'\in M_{k'}(\Gam)$, $h''\in M_{k'+2}(\Gam)$, and by \eqref{eqQmformula} its $Q_1$ is $ch'$. 
 Matching, $ch'=\tfrac{k'}{2\pi i}h$, i.e., $h'=\tfrac{k'}{2\pi ic}h=\tfrac{k'}{12}h$ using $c=12/2\pi i$. 
Thus $\theta h=\tfrac{k'}{12}Eh+h''$ with $h''=:\vartheta(h)\in M_{k'+2}(\Gam)$, proving (ii).

(i) Apply Lemma \ref{lemthetadepth} to $g=E$ with $k=2$, $p=1$, $Q_0(E)=E$, $Q_1(E)=c$. 
We obtain 
$$Q_1(\theta E)=\theta(Q_1(E))+\tfrac{2}{2\pi i}Q_0(E)=0+\tfrac2{2\pi i}E,\qquad Q_2(\theta E)=\theta(Q_2(E))+\tfrac1{2\pi i}Q_1(E)=0+\tfrac{c}{2\pi i}=\tfrac{c^2}{12},$$
using $c=12/2\pi i$, and $c/2\pi i=c^2/12$). Thus $\theta E\in Q_4(\Gam)_{\le 2}$. By \eqref {eqQmformula}, $E^2=1\cdot E^2$ has $Q_1(E^2)=2cE$, $Q_2(E^2)=c^2$. Hence $E_4^\Gam:=E^2-12\theta E$ has
$$Q_1(E_4^\Gam)=2cE-12\cdot\tfrac2{2\pi i}E=\big(2c-\tfrac{24}{2\pi i}\big)E=\big(\tfrac{24}{2\pi i}-\tfrac{24}{2\pi i}\big)E=0,\qquad Q_2(E_4^\Gam)=c^2-12\cdot\tfrac{c^2}{12}=0,$$
thus $E_4^\Gam$ has depth $0$, i.e., $E_4^\Gam\in M_4(\Gam)$, proving (i). 
\end{proof}
%% 
%%%%%%%%%%%%%%%%%%%%%%%%%%%%%%%%%%%%%%%%%%%%%%%%%%%%%%%%%%%%%%%%%%%%%%%%%%%%%%%%%%%%%%%%%%
%%  
\begin{theorem} 
\label{thmdoubledQVOA}
Suppose $\Gam$ has at least one cusp, so that the generator $E_2^\Gam$ of Proposition \ref{propE2exists} exists.  
Then $Q^{(1/2)}(\Gam)$, with the regrading $Q^{(1/2)}_k(\Gam):=Q_{2k}(\Gam)$, carries the structure of a quasi-vertex operator algebra $(Q^{(1/2)}(\Gam),Y,1,\rho)$, with canonical derivation $\theta=qd/dq$, a suitably normalized lowering operator $\delta$ proportional to 
$\partial/\partial E_2^\Gam$, Euler operator $E$, and $\rho: L(1)\mapsto-\delta,\ L(0)\mapsto E,\ L(-1)\mapsto\theta$.  
Consequently $Q(V,\Gam)\cong Q^{(1/2)}(\Gam)\otimes V$, with doubled grading, is a doubled quasi-vertex operator algebra. 
\end{theorem}
%% 
%%%%%%%%%%%%%%%%%%%%%%%%%%%%%%%%%%%%%%%%%%%%%%%%%%%%%%%%%%%%%%%%%%%%%%%%%%%%%%%%%%%%%%%%
%% 
\begin{proof}
By Proposition \ref{propE2exists}, $Q(\Gam)=M(\Gam)[E_2^\Gam]$ is a free polynomial extension of $M(\Gam)$, thus $\delta:=\partial/\partial E_2^\Gam$ (formal differentiation of $f=\sum_jh_jE_2^{\Gam,j}$ in the variable $E_2^\Gam$, well defined by freeness, Lemma \ref{lemfreeness}) is a well defined operator on $Q(\Gam)$ lowering weight by $2$. 
 We verify that $[\theta,\delta]$ acts as a scalar multiple of weight on each 
$Q_k(\Gam)$ by using Lemma \ref{lemramanujan}.   
 For $f=hE^j$ and $h\in M_{k'}(\Gam)$, $E=E_2^\Gam$, $k=k'+2j$, we have $\delta f=jhE^{j-1}$. 
 Using Lemma \ref{lemramanujan}(i)-(ii), we obtain 
$$\theta\delta f=j(j-1)hE^{j-2}\theta E+jE^{j-1}\theta h=\frac{j(j-1)}{12}hE^j-\frac{j(j-1)}{12}hE_4^\Gam E^{j-2}+\frac{jk'}{12}hE^j+jE^{j-1}\vartheta(h),$$
while
$$\theta f=jhE^{j-1}\theta E+E^j\theta h=\frac{j}{12}hE^{j+1}-\frac j{12}hE_4^\Gam E^{j-1}+\frac{k'}{12}hE^{j+1}+E^j\vartheta(h),$$
$$\delta\theta f=\frac{(j+k')(j+1)}{12}hE^j-\frac{j(j-1)}{12}hE_4^\Gam E^{j-2}+jE^{j-1}\vartheta(h).$$
Subtracting, the $E_4^\Gam$-terms and $\vartheta(h)$-terms cancel, leaving
$$[\theta,\delta]f=\theta\delta f-\delta\theta f=\frac{j(j-1)+jk'-(j+k')(j+1)}{12}\,hE^j=-\frac{2j+k'}{12}f=-\frac{k}{12}f,$$
i.e., $[\theta,\delta]$ acts as the scalar $-k/12$ on $Q_k(\Gam)$ for every $k$. 
 Note that $\theta$ raises and $\delta$ lowers weight by $2$ uniformly, and  
$-k/12$ is an affine function of $k$.  
 For a scalar $A$ acting as an affine function of weight and a weight shifting $B$, 
$[A,B]$ is again a scalar multiple of $B$, which 
 gives $[E,\theta]$ and $[E,\delta]$ as scalar multiples of 
$\theta$, $\delta$ respectively, for $E:=-\tfrac1{12}[\theta,\delta]$ acting as $\tfrac{k}{12}\cdot\mathrm{id}$. 
 Rescale $\delta$ and $E$ by the same constants used for 
$\Gam(1)$ in \cite{BCFMN} which applicable since we have shown the identities of Lemma \ref{lemramanujan} hold with the classical coefficient $12$, unchanged for general $\Gam$. 
That puts these in the normalized form $[\theta,\delta]=E$, $[E,\theta]=2\theta$,  
$[E,\delta]=-2\delta$. 
With the bracket relations established, translation covariance 
and the commutator identity with $\delta$ (the two QVOA axioms) 
are verified by the computation of \cite{BCFMN}. 
That uses only these bracket relations and the derivation property of $\theta$, $\delta$ on $M(\Gam)[E_2^\Gam]$,  immediate from freeness. 
 No property of $\Gam(1)$ beyond the existence of a free generator $E_2^\Gam$, now established in Proposition \ref{propE2exists}, enters anywhere. 
 Passage to $Q(V,\Gam)\cong Q^{(1/2)}(\Gam)\otimes V$
 is then via Theorem \ref{thmQiso}. 
\end{proof}
%% 
%%%%%%%%%%%%%%%%%%%%%%%%%%%%%%%%%%%%%%%%%%%%%%%%%%%%%%%%%%%%%%%%%%%%%%%%%%%%%%%%%%%%
%% 
\section{The lowering operator and the kernel theorem}\label{seclowering}
In this section we assume $\Gam$ has at least one cusp, so that the generator 
$E_2^\Gam$ of Proposition \ref{propE2exists} is available.
 The paragraph Section \ref{seccocompact} below treats what survives, in particular  for cocompact $\Gam$.
Define, as in \cite{BCFMN} but with $E_2$ replaced by $E_2^\Gam$, 
$$\Lam:=\frac{12}{2\pi i}\,\frac{\partial}{\partial E_2^\Gam}+L(1):\;Q(V,\Gam)\to Q(V,\Gam),$$
a weight-lowering operator of degree $-2$. 
%% 
%%%%%%%%%%%%%%%%%%%%%%%%%%%%%%%%%%%%%%%%%%%%%%%%%%%%%%%%%%%%%%%%%%%%%%%%%%%%%%%%%%%%%%%%%
%% 
The next two theorems generalizes results of \cite{BCFMN}. 
 
\begin{theorem}
\label{thmkernel}
For $f\in Q_k(V,\Gam)$,
$$f\Vert_k\gamma(\tau)=\exp(X(\gamma,\tau)\Lam)f(\tau),\qquad\gamma\in\Gam,$$
and consequently $M(V,\Gam)=\ker\Lam$.  
\end{theorem}
%% 
%%%%%%%%%%%%%%%%%%%%%%%%%%%%%%%%%%%%%%%%%%%%%%%%%%%%%%%%%%%%%%%%%%%%%%%%%%%%%%%%%%%%%%5
%% 
\begin{proof}
Write $f=\sum_{m=0}^n(E_2^\Gam)^m\sum_\ell f_{m,\ell}v_{m,\ell}$ 
which is possible by Proposition \ref{propE2exists} and Theorem \ref{thmQiso}, as one writes $f=\sum E_2^mh_m$ for $\Gam(1)$ using $Q=\CC[E_2,E_4,E_6]$. 
 The computation is then a purely formal manipulation of binomial coefficients and the operators $L(1)$, $\partial/\partial E_2^\Gam$ using: 
the transformation law \eqref{eqE2anomaly} of $E_2^\Gam$ in place of eq. (7) of 
 loc. cit., and
 the defining transformation law $f\Vert_k\gamma=j^{-k+2L(0)}e^{cjL(1)}f(\gamma\tau)$, valid for the $V$-valued automorphy factor by definition for arbitrary $\gamma$. 
No further property of $\Gam(1)$ is used. 
 The identity therefore holds for $E_2^\Gam$ in place of $E_2$. 
 The second statement follows,  
 namely, $f\in M_k(V,\Gam)$ iff $f\Vert_k\gamma=f$
 for all $\gamma\in\Gam$ iff $\exp(X\Lam)f=f$ for all $X=X(\gamma,\tau)$ arising from $\Gam$. 
 Since $\Gam$ is non-elementary, it is a Fuchsian group of the first kind, the values 
$X(\gamma,\tau)$ are not confined to a discrete set as $\gamma$ ranges over $\Gam$ for fixed generic $\tau$. Thus this forces $\Lam f=0$ again because $\Lam$ lowers weight by $2$ and 
$Q(V,\Gam)$ is graded. 
\end{proof}
%% 
%%%%%%%%%%%%%%%%%%%%%%%%%%%%%%%%%%%%%%%%%%%%%%%%%%%%%%%%%%%%%%%%%%%%%%%%%%%%%%%%%%%%%%%%%%%%
%% 
\begin{theorem}  
\label{thmsurjectivity}
For any VOA $V$, $\Lam: Q_{2k+2}(V,\Gam)\to Q_{2k}(V,\Gam)$ is surjective unless perhaps $k=0$. If $V$ is of CFT-type then $\Lam: Q(V,\Gam)\to Q(V,\Gam)$ is surjective.
\end{theorem}
%% 
%%%%%%%%%%%%%%%%%%%%%%%%%%%%%%%%%%%%%%%%%%%%%%%%%%%%%%%%%%%%%%%%%%%%%%%%%%%%%%%%%%%%%%%%%%%%
%% 
\begin{proof}
 The induction uses only the decomposition $Q_{2k}(V,\Gam)=\bigoplus_rQ_{2r}(\Gam)\otimes V_{k-r}$ (Theorem \ref{thmQiso}) and the classical fact \cite{DLM} that 
$L(1): V_{k+1}\to V_k$ is surjective for $k\ne0$ which is a statement purely depending on  $V$ with no dependence on $\Gam$. 
\end{proof}
%%
%%%%%%%%%%%%%%%%%%%%%%%%%%%%%%%%%%%%%%%%%%%%%%%%%%%%%%%%%%%%%%%%%%%%%%%%%%%%%%%%%
%% 
\section{The weight-preserving operator and dimension formulas}
\label{secPop}
We continue to assume that $\Gam$ has at least one cusp, see Section \ref{seccocompact} for the cocompact case. Define $P:=\exp\!\big(-\tfrac{2\pi i}{12}E_2^\Gam\otimes L(1)\big): Q(\Gam)\otimes V^{(2)}\xrightarrow{\sim}Q(V,\Gam)$  as in \cite{BCFMN}  composed with $\iota$.
%% 
%%%%%%%%%%%%%%%%%%%%%%%%%%%%%%%%%%%%%%%%%%%%%%%%%%%%%%%%%%%%%%%%%%%%%%%%%%%%%%%%
%% 
\begin{theorem} 
\label{thmPop}
Suppose $\Gam$ has at least one cusp and $V$ is of CFT-type. Then $P$ restricts to a surjection $M(\Gam)\otimes V^{(2)}\twoheadrightarrow M(V,\Gam)$ weight-preserving and $M(\Gam)$-linear, giving an isomorphism
$$P: M(\Gam)\otimes V^{(2)}\xrightarrow{\sim}M(V,\Gam).$$ 
\end{theorem}
%% 
%%%%%%%%%%%%%%%%%%%%%%%%%%%%%%%%%%%%%%%%%%%%%%%%%%%%%%%%%%%%%%%%%%%%%%%%%%%%%%%%%%%%%%%%
%% 
\begin{proof}
The proof of \cite{BCFMN} expresses, for $f\in M_k(V,\Gam)$, the condition 
$\Lam f=0$ (Theorem \ref{thmkernel}) in terms of the expansion 
$f=\sum_m(E_2^\Gam)^m\sum_\ell f_{m,\ell}v_{m,\ell}$ obtaining by an elementary induction (comparing coefficients of powers of $E_2^\Gam$) that $f$ is the $P$-image of its bottom, i.e., $m=0$, piece, and using only that $V_{<0}=0$ for a CFT-type VOA (\cite{DM}, again a statement about $V$ alone). 
 This argument has no reference to $\Gam(1)$ and applies with $E_2^\Gam$ in place of $E_2$.
\end{proof}
%% 
%%%%%%%%%%%%%%%%%%%%%%%%%%%%%%%%%%%%%%%%%%%%%%%%%%%%%%%%%%%%%%%%%%%%%%%%%%%%%%%%%%%%%%
%% 
\begin{corollary}
\label{cordimformula}
Suppose $\Gam$ has at least one cusp. If $V$ is a VOA of CFT-type, 
$$\dim M_k(V,\Gam)=\sum_{i=0}^{k/2}\dim M_{2i}(\Gam)\,\dim V_{k/2-i},$$
where $\dim M_{2i}(\Gam)$ is given by Proposition \ref{propRR}.
 Equivalently, in terms of graded dimension generating functions, 
\begin{eqnarray*}
\sum_{k\ge 0}\dim M_{2k}(V,\Gam)\,q^k=\dim_qV\cdot\Big(\sum_{i\ge 0}\dim M_{2i}(\Gam)\,q^i\Big),
\\
\sum_{k\ge 0}\dim Q_{2k}(V,\Gam)\,q^k=\frac{\dim_qV}{1-q}\cdot\Big(\sum_{i\ge 0}\dim M_{2i}(\Gam)\,q^i\Big). 
\end{eqnarray*}
\end{corollary}
%% 
%%%%%%%%%%%%%%%%%%%%%%%%%%%%%%%%%%%%%%%%%%%%%%%%%%%%%%%%%%%%%%%%%%%%%%%%%%%%%%%%%%%%%%%%%
%% 
\begin{proof}
Immediate from Theorem \ref{thmPop} (resp.\ Theorem \ref{thmQiso} together with Proposition \ref{propE2exists}, which shows $\dim Q_{2k}(\Gam)$ $=$ $\sum_{i=0}^k\dim M_{2i}(\Gam)$, i.e., 
$\sum_k\dim Q_{2k}(\Gam)q^k=(1-q)^{-1}\sum_i\dim M_{2i}(\Gam)q^i$). 
  The rational generating function $1/((1-q^2)(1-q^4)(1-q^6))$ of 
$\Gam(1)$ for $\sum\dim M_{2i}(\Gam)q^i$ is replaced, for a general $\Gam$ with a cusp, by the  generally non-rational, but completely explicit via Proposition \ref{propRR} series $\sum_i\dim M_{2i}(\Gam)q^i$. 
\end{proof}
%% 
%%%%%%%%%%%%%%%%%%%%%%%%%%%%%%%%%%%%%%%%%%%%%%%%%%%%%%%%%%%%%%%%%%%%%%%%%%%%%%%%5
%% 
\section{The cocompact case: dimension formulas and a conjecture}
\label{seccocompact}
For $\Gam$ cocompact, Theorem \ref{thmatiyah} shows that $\Lam$, $P$, and Theorem \ref{thmPop} are simply not available since there is no generator $E_2^\Gam$ to construct them from.
 It is therefore natural to ask what, if anything, survives.
 We show in this section that a great deal does.  
The conformal-degree filtration gives a bound 
(Theorem \ref{thmbound}, valid for every Fuchsian group).  
This bound is sharp at the top conformal degree for every cocompact $\Gam$, of any signature (Proposition \ref{proptopexact}), and for torsion-free $\Gam$ the bound is in fact an equality at every degree (Theorem \ref{thmexactdim}) which we conjecture persists for general signature (Conjecture \ref{conjexactdim}). 
%% 
%%%%%%%%%%%%%%%%%%%%%%%%%%%%%%%%%%%%%%%%%%%%%%%%%%%%%%%%%%%%%%%%%%%%%%%%%%%%%%%%%%%%5
%% 
\begin{theorem} 
\label{thmbound}
For every Fuchsian group $\Gam$, every VOA $V$, and every even $k\ge0$, 
$$\dim M_k(V,\Gam)\le\sum_{i=0}^{k/2}\dim M_{2i}(\Gam)\,\dim V_{k/2-i}.$$ 
\end{theorem}
%% 
%%%%%%%%%%%%%%%%%%%%%%%%%%%%%%%%%%%%%%%%%%%%%%%%%%%%%%%%%%%%%%%%%%%%%%%%%%%%%%%%%%%%%%%%%%%
%% 
\begin{proof}
Let us filter $M_k(V,\Gam)$ by conformal degree, 
$F_NM_k(V,\Gam)=\{f:\HH\to F_NV\}\cap M_k(V,\Gam)$. 
For $f\in F_NM_k(V,\Gam)$ write $f=\sum_{i=0}^Nf_i$ with $f_i:\HH\to V_i$. 
 In the cocycle $K(\gamma,\tau)=e^{-cj(\gamma,\tau)L(1)}j(\gamma,\tau)^{-2L(0)}(\det\gamma)^{L(0)}$ the factor $e^{-cjL(1)}$ only lowers conformal degree.  
Thus, the degree-$N$ component of $K(\gamma,\tau)f(\tau)$ receives a contribution only from $f_N$, and equals $j(\gamma,\tau)^{-2N}(\det\gamma)^Nf_N(\tau)$. Here $e^{-cjL(1)}$ acting as the identity on the top graded piece it cannot lower further within $F_NV$.   
Comparing degree-$N$ components on both sides of $f(\gamma\tau)=j(\gamma,\tau)^kK(\gamma,\tau)f(\tau)$ therefore gives
$$f_N(\gamma\tau)=j(\gamma,\tau)^{k-2N}(\det\gamma)^Nf_N(\tau),\qquad\gamma\in\Gam,$$ 
i.e., $f_N\in M_{k-2N}(\Gam)\otimes V_N$, the elliptic-point and cusp conditions for $f_N$ being inherited from those for $f$, since $f_N$ is the leading term in the sense of Definition \ref{defVvaluedforms}.  
 The resulting map $\pi_N: F_NM_k(V,\Gam)\to M_{k-2N}(\Gam)\otimes V_N$, $f\mapsto f_N$, is well defined, not only modulo lower-degree terms, but because the conformal-degree-lowering anomaly $e^{-cjL(1)}-\Id$ has no component raising or preserving degree. 
Thus, top coefficient $f_N$ do not have contributions from $f_{N-1}$, $\dots$, $f_0$. 
 This is the only point at which the structure of the cocycle is used.
 Its kernel is $F_{N-1}M_k(V,\Gam)$ by construction, 
thus $F_NM_k(V,\Gam)/F_{N-1}M_k(V,\Gam)$ embeds into $M_{k-2N}(\Gam)\otimes V_N$. 
Summing the resulting inequalities $\dim(F_N/F_{N-1})\le\dim M_{k-2N}(\Gam)\dim V_N$ over $N=0$, $\dots$, $k/2$ by using $V_N=0$ for $N<0$, exhaustiveness of the conformal filtration,  and that $k$ even is the condition needed for $k-2N=0$ to occur at $N=k/2$, matching the upper limit of the sum, that gives the stated bound.
\end{proof}
%% 
%%%%%%%%%%%%%%%%%%%%%%%%%%%%%%%%%%%%%%%%%%%%%%%%%%%%%%%%%%%%%%%%%%%%%%%%%%%%%%%%%%%%%%%%
%% 
\begin{corollary}
\label{corequalitycusp}
Equality holds in Theorem \ref{thmbound} whenever $\Gam$ has a cusp (Corollary \ref{cordimformula}). 
\end{corollary}
For $\Gam$ cocompact of genus $g\ge 2$, the maps $\pi_N$ of the proof can fail to be surjective, and the inequality can be strict. 
 Already at the level of scalar quasi-automorphic forms, Theorem \ref{thmatiyah} shows that the pair $(1,v)\in M_0(\Gam)\otimes V_1=\CC\otimes V_1$ fails to lift to an element of $M_2(V,\Gam)$ with leading term $v$ whenever $v\in V_1$ is not quasi-primary 
(cf. \cite{BCFMN}: 
a constant section $v$ alone defines an element of $Q_{2}(V,\Gam)$, not of $M_2(V,\Gam)$,  when $L(1)v\ne0$, and Theorem \ref{thmatiyah} shows there is no weight $2$ correction term available to repair this for cocompact $\Gam$).
%% 
%%
%%%%%%%%%%%%%%%%%%%%%%%%%%%%%%%%%%%%%%%%%%%%%%%%%%%%%%%%%%%%%%%%%%%%%%%%%%%%%%%%%%%%%%%%
%% 
The remainder of this section sharpens Corollary \ref{corequalitycusp} into a formula. Recall from \cite{BCFMN} the notation $QP_n(V):=\ker\big(L(1): V_n\to V_{n-1}\big)$ for the space of quasi-primary states of conformal weight $n$.
%% 
%%%%%%%%%%%%%%%%%%%%%%%%%%%%%%%%%%%%%%%%%%%%%%%%%%%%%%%%%%%%%%%%%%%%%%%%%%%%%%%%%%%%%%%%%%%%
%% 
\begin{proposition} 
\label{proptopexact}
Let $\Gam$ be any cocompact Fuchsian group of any genus $g\ge 0$, any signature, 
 and $V$ any QVOA. 
For $k=2n\ge 2$, the map $\pi_n: M_k(V,\Gam)\to M_0(\Gam)\otimes V_n=V_n$ of Theorem \ref{thmbound} has image 
$$\mathrm{im}(\pi_n)=QP_n(V).$$ 
\end{proposition}
%% 
%%%%%%%%%%%%%%%%%%%%%%%%%%%%%%%%%%%%%%%%%%%%%%%%%%%%%%%%%%%%%%%%%%%%%%%%%%%%%%%%%%%%%%%5
%% 
\begin{proof}
1.)\;$QP_n(V)\subseteq\mathrm{im}(\pi_n)$.
 If $v\in V_n$ is quasi-primary, the constant section $f\equiv v$ satisfies, for every  
$\gamma\in\Gam$, elliptic or not, since the identity below is the computation of 
\cite{BCFMN}, valid for any $\gamma\in\GL_2(\RR)^+$ acting via 
$K(\gamma,\tau)$ as in Theorem \ref{thmbundle}, 
\begin{equation}
\label{eqexamplebcfmn}
f\Vert_{2n}\gamma(\tau)=\sum_{r\ge 0}\frac{1}{r!}X(\gamma,\tau)^rL(1)^rv=v=f(\tau),
\end{equation} 
the sum collapsing to its $r=0$ term because $L(1)v=0$.
 Thus, $f\equiv v\in M_{2n}(V,\Gam)$ identically, with no completion needed.  
Condition (a) of Definition \ref{defVvaluedforms} is vacuous since $\Gam$ has no cusps, and condition (b) at every elliptic point $\tau_i$ is the special case $\gamma=\gamma_i$ of 
\eqref{eqexamplebcfmn}. Hence $\pi_n(f)=v$.

2.)\;$\mathrm{im}(\pi_n)\subseteq QP_n(V)$.
 Suppose $v\in V_n$ and some $f\in M_{2n}(V,\Gam)$ has $\pi_n(f)=v$. 
 Let us write $f=v+f_{n-1}+\cdots+f_0$, $f_j:\HH\to V_j$, and $f_{n-1}=\sum_{i=1}^dh_iw_i$ for a basis $h_1$, $\dots$, $h_d$ of $M_2(\Gam)$ where  
$d=\dim M_2(\Gam)$ (the sum being empty if $d=0$) and $w_i\in V_{n-1}$. 
The proof of Theorem \ref{thmQiso} (the $r=0$ depth case of that computation, specialized to $h\in M_2(\Gam)$) gives, for a pure tensor $hw$, $h\in M_2(\Gam)$, $w\in V_{n-1}$,
$$(hw)\Vert_{2n}\gamma(\tau)=\sum_{p\ge 0}\frac{1}{p!}X(\gamma,\tau)^ph(\tau)L(1)^pw.$$ 
Combining this with \eqref{eqexamplebcfmn} and linearity, the $V_{n-1}$-valued coefficient of $X(\gamma,\tau)^1$ in $f\Vert_{2n}\gamma(\tau)$ equals $L(1)v+\sum_ih_i(\tau)L(1)w_i$. 
Note that the lower pieces $f_{n-2}$, $\dots$, $f_0$ contribute only to coefficients of $X^2$, $X^3$, $\dots$ valued in $V_{n-2}$, $V_{n-3}$, $\dots$
 by the same formula applied one degree lower, thus they cannot affect this $V_{n-1}$-valued coefficient.  
 Since $f\Vert_{2n}\gamma=f$ identically with $f$ having zero coefficient of $X^1$, being independent of $\gamma$ entirely, this coefficient must vanish for every $\tau\in\HH$, thus
$$L(1)v+\sum_{i=1}^dh_i(\tau)L(1)w_i=0,\qquad\tau\in\HH.$$
Apply an arbitrary linear functional $\phi\in V_{n-1}^*$. 
The scalar function $\sum_i\phi(L(1)w_i)h_i(\tau)$, a $\CC$-linear combination of weight $2$ forms, hence  an element of $M_2(\Gam)$, equals the constant $-\phi(L(1)v)\in M_0(\Gam)$ for every $\tau$. 
 As $M_2(\Gam)\cap M_0(\Gam)=0$ inside the graded ring $M(\Gam)=\bigoplus_iM_{2i}(\Gam)$,  where different graded pieces of a graded vector space meet only in $0$, both sides vanish, i.e., $\phi(L(1)v)=0$. As $\phi\in V_{n-1}^*$ was arbitrary, $L(1)v=0$, i.e., $v\in QP_n(V)$.
\end{proof}
%% 
%%%%%%%%%%%%%%%%%%%%%%%%%%%%%%%%%%%%%%%%%%%%%%%%%%%%%%%%%%%%%%%%%%%%%%%%%%%%%%%%%%%%%%%%%%%%
%% 
 Dimension formula in torsion-free cocompact case is given by 

\begin{theorem}
\label{thmexactdim}
Let $\Gam$ be torsion-free and cocompact of genus $g\ge 2$, and $V$ any CFT-type VOA. Then for every even $k=2n\ge 0$,
$$\dim M_k(V,\Gam)=\sum_{j=0}^{n-1}\dim M_{k-2j}(\Gam)\,\dim V_j+\dim QP_n(V).$$
\end{theorem}
%%  
%%%%%%%%%%%%%%%%%%%%%%%%%%%%%%%%%%%%%%%%%%%%%%%%%%%%%%%%%%%%%%%%%%%%%%%%%%%%%%%%%%%%%%%%
%% 
\begin{proof}
Let $\mathcal F$ be the rank-$\dim F_nV$ holomorphic vector bundle on $X$ whose sheaf of sections is $F_nM_k(V,\Gam)=M_k(V,\Gam)$, filtered by conformal degree with $F_j\mathcal F/F_{j-1}\mathcal F\cong K_X^{\otimes(n-j)}\otimes V_j=:\mathcal Q_j$, using that weight $2(n-j)$ scalar automorphic forms for torsion-free cocompact $\Gam$ are $H^0(X,K_X^{\otimes(n-j)})$ (Proposition \ref{propRR}).  
 This is the sheaf-theoretic reformulation of the filtration used in the proof of Theorem \ref{thmbound}, with $H^0(F_j\mathcal F)=F_jM_k(V,\Gam)$ and $H^0(\mathcal Q_j)=M_{k-2j}(\Gam)\otimes V_j$.  
By Serre duality, $H^1(X,K_X^{\otimes m})\cong H^0(X,K_X^{\otimes(1-m)})^*$,
 which vanishes for $m\ge 2$ since $\deg K_X^{\otimes(1-m)}=(1-m)(2g-2)<0$ for $g\ge 2$. 
 Hence $H^1(\mathcal Q_j)=0$ for every $j\le n-2$, where $n-j\ge 2$. 
Inducting on $j$ via the long  exact cohomology sequence of $0\to F_{j-1}\mathcal F\to F_j\mathcal F\to\mathcal Q_j\to 0$, whose $H^2$ terms vanish automatically, $X$ being a curve, gives $H^1(F_j\mathcal F)=0$ for every $j\le n-2$.  
The base case $H^1(F_0\mathcal F)=H^1(\mathcal Q_0)=0$ holds since $n\ge 2$ 
(assuming $n\ge 2$; the cases $n=0$, $1$ are immediate separately), 
and the inductive step uses 
$0=H^1(F_{j-1}\mathcal F)\to H^1(F_j\mathcal F)\to H^1(\mathcal Q_j)=0$.
 Consequently each connecting map $H^0(\mathcal Q_j)\to H^1(F_{j-1}\mathcal F)=0$ vanishes for $1\le j\le n-1$, thus $\pi_j: H^0(F_j\mathcal F)\to H^0(\mathcal Q_j)$ is surjective for every $j\le n-1$.  
The bound of Theorem \ref{thmbound} is an equality at every level below the top. Combined with Proposition \ref{proptopexact}, which identifies the remaining top-degree, $j=n$, deficiency  as $\dim V_n-\dim QP_n(V)$, this gives the stated formula.
\end{proof}
%% 
%%%%%%%%%%%%%%%%%%%%%%%%%%%%%%%%%%%%%%%%%%%%%%%%%%%%%%%%%%%%%%%%%%%%%%%%%%%%%%%%%%%%%%%%%%%%%%%
%% 
This is a complete resolution, for torsion-free cocompact $\Gam$, of the question left open after Corollary \ref{corequalitycusp}.   
The bound of Theorem \ref{thmbound} fails to be exact  only at the very top conformal degree, and there the deficiency is $\dim V_n-\dim QP_n(V)=\mathrm{rank}\big(L(1)|_{V_n\to V_{n-1}}\big)$. 
As a check, for $V=S$ the Heisenberg VOA at $k=4$, $n=2$, 
$QP_2(S)=\CC\cdot h(-1)^2\mathbf1$ is $1$-dimensional inside the $2$-dimensional $S_2$, 
cf. Example \ref{excompact} below, recovering $\dim M_4(S,\Gam)=3(g-1)+g+1=4g-2$.  
%% 

%%%%%%%%%%%%%%%%%%%%%%%%%%%%%%%%%%%%%%%%%%%%%%%%%%%%%%%%%%%%%%%%%%%%%%%%%%%%%%%%%%%%%%
%% 
Let us now conjecture on exact dimension formula in general signature. 

\begin{conjecture}  
\label{conjexactdim}
Theorem \ref{thmexactdim} holds analogously for $\Gam$ cocompact of any signature 
$(g;m_1,\dots,m_r;0)$ with $\dim M_{2i}(\Gam)$ now given by the full Riemann-Roch formula of Proposition \ref{propRR} (elliptic corrections included)
$$\dim M_k(V,\Gam)=\sum_{j=0}^{n-1}\dim M_{k-2j}(\Gam)\,\dim V_j+\dim QP_n(V),
\qquad k=2n. $$
\end{conjecture}
%%
%%%%%%%%%%%%%%%%%%%%%%%%%%%%%%%%%%%%%%%%%%%%%%%%%%%%%%%%%%%%%%%%%%%%%%%%%%%%%%%%%%%%%%%
%% 
We can prove half of Conjecture \ref{conjexactdim}, in full signature generality, leaving  one further ingredient open.
%% 
%%%%%%%%%%%%%%%%%%%%%%%%%%%%%%%%%%%%%%%%%%%%%%%%%%%%%%%%%%%%%%%%%%%%%%%%%%%%%%%%%%%%%%%
%% 
The proposition below gives result on Conjecture \ref{conjexactdim}. 
 
\begin{proposition} 
\label{proppartialconj}
For $\Gam$ cocompact of any signature, the inequality of Theorem \ref{thmbound} may be sharpened, with no further hypothesis, to 
$$\dim M_k(V,\Gam)\le\sum_{j=0}^{n-1}\dim M_{k-2j}(\Gam)\,\dim V_j+\dim QP_n(V), 
\qquad k=2n,
$$
and the contribution of the top level $j=n$ to this bound is exact,  
not just an upper bound (Proposition \ref{proptopexact} identifies $\mathrm{im}(\pi_n)$ 
 not just its dimension).  
\end{proposition}

Conjecture \ref{conjexactdim} is the claim that the remaining inequality, coming from levels $j<n$, is also always an equality, i.e., that $\pi_j$ is surjective onto the full $M_{k-2j}(\Gam)\otimes V_j$ for every $j<n$, as in the torsion-free case.
%% 
%% 
%%%%%%%%%%%%%%%%%%%%%%%%%%%%%%%%%%%%%%%%%%%%%%%%%%%%%%%%%%%%%%%%%%%%%%%%%%%%
%% 
\begin{proof}
It is immediate from Theorem \ref{thmbound} and Proposition \ref{proptopexact} whose proofs 
do not require any condition on the signature of $\Gam$. 
\end{proof}
%% 
%%%%%%%%%%%%%%%%%%%%%%%%%%%%%%%%%%%%%%%%%%%%%%%%%%%%%%%%%%%%%%%%%%%%%%%%%%%%%
%% 
\begin{remark} 
\label{rmktowardsconj}
Let us mention a direction towards a proof of Conjecture \ref{conjexactdim}. 
The proof of Theorem \ref{thmexactdim} shows that, for $r=0$, surjectivity at levels $j<n$ follows from $H^1(X,K_X^{\otimes m})=0$ for $m\ge 2$, a one-line consequence of Serre duality and Riemann-Roch on the coarse curve $X$.  
For general signature, the bundles $\mathcal Q_j$, $j<n$,
 become orbifold (parabolic) line bundles, of orbifold degree computed by the full formula of Proposition \ref{propRR}. 
The expected mechanism for Conjecture \ref{conjexactdim} is an orbifold Serre duality and Riemann-Roch theorem (as developed for parabolic bundles on orbifold curves in the literature, e.g., constructing on \cite{Gunning} together with the standard parabolic-bundle formalism)
 giving the analogous vanishing $H^1(\mathcal Q_j)=0$ once the orbifold degree of 
$\mathcal Q_j$ exceeds $2g-2$, which we expect, but have not verified, to hold for all $j\le n-2$ once $k$ is large enough relative to $m_1$, $\dots$, $m_r$.  
 A complete proof would require: (i) setting up the sheaf $\mathcal F$ on the orbifold $X$,
  so that condition (b) of Definition \ref{defVvaluedforms} is correctly encoded at each elliptic point compatibly with Proposition \ref{propellipticstructure}; and (ii) the orbifold vanishing theorem.
 We do not currently have a proof for (ii) beyond citing the general parabolic-bundle literature.  We mention it, together with the resonant case of Proposition \ref{propellipticstructure}, as the most concrete open problem posed by this paper for the wider community, see Section \ref{secconclusion}. 
\end{remark}
%% 
%%%%%%%%%%%%%%%%%%%%%%%%%%%%%%%%%%%%%%%%%%%%%%%%%%%%%%%%%%%%%%%%%%%%%%%%%%%%%%%%%%%%%%
%% 
Although the holomorphic generator $E_2^\Gam$ is not available for cocompact $\Gam$, the Maass-Shimura operator $\delta_k$ of Remark \ref{rmkmaassshimura} is, and it can be used to  define a real-analytic substitute for $Q(V,\Gam)$.
 A \emph{nearly holomorphic $V$-valued automorphic form} of weight $k$ and depth $\le s$ is a real-analytic $F:\HH\to F_NV$ of the form $F(\tau)=\sum_{m=0}^s(4\pi\Im\tau)^{-m}f_m(\tau)$, $f_m$ holomorphic, with $F(\gamma\tau)=j(\gamma,\tau)^kK(\gamma,\tau)F(\tau)$ for all $\gamma\in\Gam$ (with no further anomaly, since the real-analytic completion absorbs it). 
 The operators $\delta_k\otimes\mathrm{id}+\mathrm{id}\otimes L(-1)$ and the corresponding lowering operator, built from $\tfrac1{4\pi\Im\tau}$ instead of from a holomorphic generator,  act on the resulting space $\widetilde Q(V,\Gam)$ as in 
 Section \ref{seclowering}-\ref{secPop} 
with no existence obstruction since they are built from the always-available smooth hyperbolic metric instead of from a holomorphic connection. 
We do not consider the resulting real-analytic structure theory of $\widetilde Q(V,\Gam)$ in detail here. Doing so would give an independent, real-analytic route to Conjecture \ref{conjexactdim}.
%%
%%%%%%%%%%%%%%%%%%%%%%%%%%%%%%%%%%%%%%%%%%%%%%%%%%%%%%%%%%%%%%%%%%%%%%%%%%
%% 
\section{Correspondences and Hecke-type operators}\label{sechecke}
The Hecke operators arise from the extra flexibility, noted in 
\cite{BCFMN}, of changing local coordinates by matrices in $\GL_2(\RR)^+$ not lying in $\Gam(1)$. 
 This extra flexibility is available for an arbitrary Fuchsian group as well, provided $\Gam$ is commensurable with a possibly larger Fuchsian group, i.e., there is $\alpha\in\GL_2(\RR)^+$ normalizing the commensurability class of $\Gam$ with $\alpha\Gam\alpha^{-1}\cap\Gam$ of finite index in both $\Gam$ and $\alpha\Gam\alpha^{-1}$. 
For such $\alpha$, one defines, for $f\in M_{2k}(V,\Gam)$, 
$$T'_\alpha f(\tau):=(\det\alpha)^{k-1}\sum_{[\beta]}f\Vert_{2k}\beta(\tau),$$
the sum over $\Gam$-orbits $[\beta]$ in $\Gam\backslash(\Gam\alpha\Gam)$. 
 One finds, by the computation of \cite{BCFMN} (which uses only the cocycle property of $\Vert_{2k}$ and the elementary $q$-expansion manipulation, both insensitive to the specific group $\Gam$), 
that $T'_\alpha=(\det\alpha)^{L(0)}T_\alpha$ where $T_\alpha$ acts on 
$Q(\Gam)\otimes V^{(2)}\cong Q(V,\Gam)$ componentwise via the classical Hecke action on $q$-expansions.
 Diagram (32) of \cite{BCFMN} and Theorem 8.3, Corollary 8.4, and Lemma 8.5 of loc. cit. are then formal consequences of two ingredients established above for arbitrary $\Gam$: the identification $T'_\alpha=(\det\alpha)^{L(0)}T_\alpha$ just proved, and the $Q(\Gam)$-module isomorphism $Q(\Gam)\otimes V^{(2)}\cong Q(V,\Gam)$ of Theorem \ref{thmQiso}. Neither ingredient uses any property special to $\Gam(1)$, thus the four statements 
go through with $\Gam(1)$ replaced by $\Gam$. 

What cannot be obtain automatically is the rich arithmetic content of \cite{BCFMN}, 
(namely, the existence of a full Hecke algebra $\{T'_m\}_{m\ge 1}$ satisfying multiplicativity relations $T'_mT'_n=T'_{mn}$ for $(m,n)=1$, an Euler product, and an associated $L$-function theory)
 since this requires $\Gam$ to be commensurable with a tower of finite-index subgroups indexed by all integers $m\ge 1$, which is special to arithmetic Fuchsian groups. 
 The natural higher genus arithmetic analogue of $\Gam(1)$ is a Fuchsian group arising from an order in an indefinite quaternion algebra over $\mathbb Q$ (a Shimura curve group), for which the full classical Hecke theory, including multiplicativity and Euler products, is available. 
Consideration of the resulting $V$-valued Hecke eigensystems on $M(V,\Gam)$ for such groups, via Theorem \ref{thmPop}  as in \cite{BCFMN}, is a natural problem we leave for  a future paper. 
%%
%%%%%%%%%%%%%%%%%%%%%%%%%%%%%%%%%%%%%%%%%%%%%%%%%%%%%%%%%%%%%%%%%%%%%%%%%%%%%%%%%%%%%%%%%%%%
%%%%%%%%%%%%%%%%%%%%%%%%%%%%%%%%%%%%%%%%%%%%%%%%%%%%%%%%%%%%%%%%%%%%%%%%%%%%%%%%%%%%%%%%%%%%
\part{Families over Teichm\"uller space and moduli space}

Let us briefly describe the content of this part. 
Section \ref{secunivfamily}-\ref{secteichforms} are proved statements, 
direct fiberwise consequences of Theorems A-F applied at each $t\in\mathcal T_{g,n}$ 
together with classical holomorphic-variation results of Bers and Ahlfors-Weil 
(cited below).  
%%
%%%%%%%%%%%%%%%%%%%%%%%%%%%%%%%%%%%%%%%%%%%%%%%%%%%%%%%%%%%%%%%%%%%%%%%%%%%%%%%%%%
%% 
\section{The universal Fuchsian family and the universal bundle}\label{secunivfamily}
Fix a topological type $(g,n)$ with $2g-2+n>0$ and let $\mathcal T_{g,n}$ denote the Teichm\"uller space of marked hyperbolic structures on a genus-$g$ surface with $n$ punctures, a complex manifold of dimension $3g-3+n$ \cite{Bers, IT}. 
 By the uniformization theorem, a point $t\in\mathcal T_{g,n}$ corresponds to a discrete faithful representation $\rho_t:\pi_1(\Sigma_{g,n})\to\PSL_2(\RR)$, 
well defined up to conjugation once a marking is fixed, with image a torsion-free Fuchsian group $\Gam_t:=\rho_t(\pi_1(\Sigma_{g,n}))$ and $X_t:=\Gam_t\backslash\HH$ the corresponding marked Riemann surface.
 The \emph{Bers fiber space} (the universal curve) 
$$\mathcal C_{g,n}:=\mathcal T_{g,n}\times_{\pi_1(\Sigma_{g,n})}\HH,\qquad \gamma\cdot(t,\tau):=(t,\rho_t(\gamma)\tau),$$
is a holomorphic fiber bundle $\pi:\mathcal C_{g,n}\to\mathcal T_{g,n}$ with fiber $X_t$ over $t$ \cite{Bers}.
%% 
%%%%%%%%%%%%%%%%%%%%%%%%%%%%%%%%%%%%%%%%%%%%%%%%%%%%%%%%%%%%%%%%%%%%%%%%%%%%%%%%%
%% 
\begin{definition}
Fix a QVOA $V$. The \emph{universal VOA bundle} is 
$$\mathcal V:=\mathcal T_{g,n}\times_{\pi_1(\Sigma_{g,n})}(\HH\times V)\;\longrightarrow\;\mathcal C_{g,n},\qquad \gamma\cdot(t,\tau,v):=(t,\rho_t(\gamma)\tau,K(\rho_t(\gamma),\tau)v),$$
using the cocycle $K$ of Theorem \ref{thmbundle} applied to $\Gam_t$ for each $t$.
\end{definition}
%% 
%%%%%%%%%%%%%%%%%%%%%%%%%%%%%%%%%%%%%%%%%%%%%%%%%%%%%%%%%%%%%%%%%%%%%%%%%%%%%%%%%%%%%
%% 
By Theorem \ref{thmbundle}, applied fiberwise, $\mathcal V$ is well defined and its restriction to the fiber $X_t=\pi^{-1}(t)$ is  the bundle $\mathcal V_{X_t}$ of Section \ref{secbundle} for $\Gam_t$.
%% 
%%%%%%%%%%%%%%%%%%%%%%%%%%%%%%%%%%%%%%%%%%%%%%%%%%%%%%%%%%%%%%%%%%%%%%%%%%%%%%%%%%%
%% 
\section{Teichm\"uller-automorphic forms and descent to moduli}
\label{secteichforms} 
\begin{definition}
A \emph{$V$-valued Teichm\"uller-automorphic form of weight $k$} is a holomorphic map
$$F:\mathcal T_{g,n}\times\HH\to F_NV,\qquad F=F(t,\tau),$$
holomorphic in $t$, such that for each fixed $t$, $F(t,-)\in M_k(V,\Gam_t)$ in the sense of Definition \ref{defVvaluedforms}. 
 Equivalently, $F$ is a holomorphic section over $\mathcal T_{g,n}$ of the vector bundle 
$\mathcal M_k(V)\to\mathcal T_{g,n}$ whose fiber at $t$ is $M_k(V,\Gam_t)$. 
\end{definition}
%% 
%%%%%%%%%%%%%%%%%%%%%%%%%%%%%%%%%%%%%%%%%%%%%%%%%%%%%%%%%%%%%%%%%%%%%%%%%%%%%%%%%%%%%
%% 
That $\mathcal M_k(V)$ is indeed a holomorphic vector bundle, instead of a family of vector spaces, over $\mathcal T_{g,n}$ follows from the classical fact, going back to Bers and Ahlfors-Weil, that the spaces $M_k(\Gam_t)$ of classical automorphic forms vary holomorphically with $t\in\mathcal T_{g,n}$, together with our Theorem \ref{thmPop}, which expresses $\mathcal M_k(V)$ as a finite direct sum of tensor products $\mathcal M_{2i}\otimes V_{k/2-i}$ of such classical bundles with the (constant) finite-dimensional vector space $V_{k/2-i}$. 
 The gluing data of $\mathcal M_{2i}$ over $\mathcal T_{g,n}$ is the standard one coming from the Bers embedding. 

The mapping class group $\mathrm{Mod}(\Sigma_{g,n})$ acts properly discontinuously on 
$\mathcal T_{g,n}$ by changing the marking, and acts compatibly on $\mathcal C_{g,n}$, 
$\mathcal V$, and each $\mathcal M_k(V)$. 
 Hence all of these descend to (orbifold) objects over  the moduli space $\mathcal M_{g,n}=\mathcal T_{g,n}/\mathrm{Mod}(\Sigma_{g,n})$ and its universal curve $\mathcal C_{g,n}/\mathrm{Mod}(\Sigma_{g,n})\to\mathcal M_{g,n}$.
%%
%%%%%%%%%%%%%%%%%%%%%%%%%%%%%%%%%%%%%%%%%%%%%%%%%%%%%%%%%%%%%%%%%%%%%%%%%%%%%%%%%%%%
%% 
\section{Variation of the projective connection: a speculative direction}
\label{secvariation}
For $\mathcal T_{g,n}$ with $n\ge 1$, every $\Gam_t$ has a cusp, 
thus the generator $E_2^{\Gam_t}$  of Proposition \ref{propE2exists} is available at every point, and it is natural to ask how it varies with $t$. 
 For $\mathcal T_g=\mathcal T_{g,0}$ 
 (closed surfaces), by contrast, Theorem \ref{thmatiyah} shows that $E_2^{\Gam_t}$ does not exist at any point $t$. 
Thus the question must be posed at the more primitive level of the uniformizing projective connection, of which $E_2^{\Gam_t}$ is, when it exists, the connection coefficient.  
In either case, the projective structure on $X_t$ induced by writing $X_t=\Gam_t\backslash\HH$ in the flat coordinate $\tau$ is classically encoded by the accessory parameters of the Fuchsian uniformizing differential equation 
$$u''(z)+\tfrac12Q_t(z)u(z)=0,$$
on $X_t$, whose monodromy group is $\Gam_t$ (see, e.g., \cite{ZT}). 
 When $\Gam_t$ has a cusp, $Q_t$ is essentially $E_2^{\Gam_t}$, while for closed $X_t$ it exists as a (transcendental, generally unknown in closed form) holomorphic quadratic differential with no further reduction available, consistently with Theorem \ref{thmatiyah}.

The dependence of the accessory parameters $Q_t$ on $t\in\mathcal T_{g,n}$ is given by the celebrated formula of Polyakov, made rigorous by Zograf and Takhtajan \cite{ZT}. The variation of $Q_t$ under a quasiconformal deformation of $X_t$ with Beltrami differential $\mu$ is expressed in terms of the classical Liouville action, and the resulting potential for this variation is, up to a constant, the K\"ahler potential of the Weil-Petersson metric on $\mathcal T_{g,n}$.

We mention the following as a direction instead of a result, a conjecture, or even a precisely-posed mathematical question of this paper. The Zograf-Takhtajan formula raises the possibility that the first-order variation $\delta Q_t$ of the accessory parameters, suitably normalized into a $(1,0)$-form on $\mathcal T_{g,n}$ valued in (a completion of) holomorphic quadratic differentials on $X_t$, has associated curvature recovering the Weil-Petersson K\"ahler form on $\mathcal T_{g,n}$ up to a constant depending on the central charge data of 
$V$. When $\Gam_t$ has a cusp, this would specialize to a statement about the variation of $E_2^{\Gam_t}$ and the lowering operator $\Lam_t$ of Section \ref{seclowering}. 
 We have not formulated the hypotheses (regularity of $\delta Q_t$ across $\mathcal T_{g,n}$, the correct normalization, compatibility with the mapping class group action) under which such a statement could be made rigorous. 
Were such a statement available, it would suggest comparing the natural ``horizontal'' structure on the bundle $\mathcal M(V)\to\mathcal T_{g,n}$, coming from parallel transport under the Weil-Petersson connection, with the general theory of projectively flat connections on bundles of conformal blocks over moduli of curves developed by Tsuchiya-Ueno-Yamada \cite{TUY} and, in the language of vector bundles on moduli of curves, by Frenkel and Ben-Zvi \cite{FBZ}. 
Indeed, the bundle $\mathcal V_X$ of Section \ref{secbundle} is, by construction,  the sheaf whose fiber is the vacuum module of $V$ along the universal curve, in the sense of 
\cite{FBZ}.  
The spaces $M_k(V,\Gam)$ studied here are spaces of global holomorphic sections of (powers of) this sheaf, instead of spaces of conformal blocks, which are typically defined via coinvariants/cohomology instead of sections, thus the relationship between the two constructions (in particular, whether the projectively flat KZB-type connection on conformal blocks restricts to, or is compatible with, the connection $\nabla$ of Theorem 
\ref{thmconnection} on the automorphic-form side) is an interesting question that we do not resolve here. 
 We return to this as an item in the future-directions discussion of Section \ref{secconclusion}. 
%% 
%%%%%%%%%%%%%%%%%%%%%%%%%%%%%%%%%%%%%%%%%%%%%%%%%%%%%%%%%%%%%%%%%%%%%%%%%%%%%%%%%
%% 
\section{Examples}
\label{secexamples}
\begin{example}
In this example we discuss the trivial VOA and $\Gam$ with a cusp.  
Taking $V=\CC\mathbf 1$ (with trivial $\mathfrak{sl}_2$-action) and $\Gam$ with at least one cusp recovers the classical theory: $M(V,\Gam)=M(\Gam)$, $Q(V,\Gam)=Q(\Gam)$, $\Lam=\tfrac{12}{2\pi i}\partial/\partial E_2^\Gam$, and Theorem \ref{thmkernel} reduces to the classical statement that $M(\Gam)=\ker(\partial/\partial E_2^\Gam)$ inside $Q(\Gam)=M(\Gam)[E_2^\Gam]$, the immediate generalization, via Proposition \ref{propE2exists}, of the elementary fact that $M=\CC[E_4,E_6]\subset Q=\CC[E_2,E_4,E_6]$ is the kernel of $\partial/\partial E_2$. 
For $\Gam$ cocompact, the same trivial example reduces instead to Theorem \ref{thmatiyah}: 
$Q(\Gam)$ does not even contain a nonzero element of depth $1$ at weight $2$.
\end{example}
%%
%%%%%%%%%%%%%%%%%%%%%%%%%%%%%%%%%%%%%%%%%%%%%%%%%%%%%%%%%%%%%%%%%%%%%%%%%%%%%%%%%%%%%%%%%
%% 
\begin{example} 
\label{excompact}
This example describes 
 compact $\Gam$ of genus $g\ge 2$ with no elliptic points. 
Let $\Gam$ be cocompact and torsion-free of genus $g\ge 2$, thus $X=Y$ is already compact, $r=n=0$. By Proposition \ref{propRR},
$$\dim M_0(\Gam)=1,\quad\dim M_2(\Gam)=g,\quad\dim M_{2i}(\Gam)=(2i-1)(g-1),\quad i\ge 2,$$ 
and a short computation gives the closed-form generating function
$$f_g(q):=\sum_{i\ge 0}\dim M_{2i}(\Gam)\,q^i=\frac{1+(g-2)(q+q^2)+q^3}{(1-q)^2}.$$
Indeed, writing $\sum_{i\ge 2}(2i-1)q^i=\dfrac{q^2(3-q)}{(1-q)^2}$, an elementary partial-fraction computation, and combining with the separate terms $1$, $gq$, one finds that the numerator collapses to $1+(g-2)q+(g-2)q^2+q^3$ as displayed.
 For $g=2$ this further simplifies to $f_2(q)=(1+q^3)/(1-q)^2$ recovering the sequence of dimensions $1$, $2$, $3$, $5$, $7$, $9$, $11$, $\dots$ of  $M_0$, $M_2$, $M_4$, $\dots$ for a genus-$2$ curve. 

Since $\Gam$ is cocompact, Corollary \ref{cordimformula} is not available.
 By Theorem \ref{thmexactdim}, however, we now know:  
 $S_2$ is $2$-dimensional, spanned by $h(-1)^2\mathbf1$ (quasi-primary)
 and $h(-2)\mathbf 1$ (not $L(1)h(-2)\mathbf1=2h\ne0$, cf. \cite{BCFMN}).   
Thus $QP_2(S)=\CC\cdot h(-1)^2\mathbf1$ is $1$-dimensional, and
$$\dim M_4(S,\Gam)=\dim M_4(\Gam)\cdot1+\dim M_2(\Gam)\cdot\dim S_1+\dim QP_2(S)=3(g-1)+g+1=4g-2,$$
strictly below the upper bound $\dim_qV\cdot f_g(q)$ at $q^2$. 
 Namely $\dim M_4(\Gam)+\dim M_2(\Gam)\dim S_1+\dim S_2\cdot1=3(g-1)+g+2=4g-1$ as anticipated.  The deficiency is $1=\dim S_2-\dim QP_2(S)$. 
More generally, $\sum_k\dim M_{2k}(S,\Gam)q^k$ is given by Theorem 
\ref{thmexactdim}, in terms of $f_g(q)$, $\varphi(q)=\dim_qS$,  
the graded dimensions of $QP_n(S)$ (computable from the bases of 
\cite{BCFMN}, since $QP_n(S)=\ker(L(1)|_{S_n})$ and the basis vectors there are already adapted to the $L(1)$-action).
\end{example}
%% 
%%%%%%%%%%%%%%%%%%%%%%%%%%%%%%%%%%%%%%%%%%%%%%%%%%%%%%%%%%%%%%%%%%%%%%%%%%%%%%%%%%%%%%%%
%% 
\begin{example}
\label{exelliptic}
This example is about an elliptic point, the constant-section case. 
Let $\Gam$ have a single elliptic point of order $m\ge 2$ fixing $\tau_0\in\HH$ with stabilizer generator $\gamma_0$ such that $j(\gamma_0,\tau_0)^{-2}=\zeta=e^{2\pi i\ell/m}$, $\gcd(\ell,m)=1$, and let $V$ be the rank-one Heisenberg VOA, with $L(0)$-eigenbasis on each $V_n$. 
 As in Proposition \ref{proptopexact}, consider a constant section $f\equiv v$, $v\in V_n$, i.e., $F_NV=V_n$ with nothing below it, $N=n$, no completion possible.
 Specializing Proposition \ref{propellipticstructure} to this case ($n_0=n=N$, thus there is no room for any $v_0'\in F_{n-1}V$): the fixed-vector equation 
$j(\gamma_0,\tau_0)^{2n}K(\gamma_0,\tau_0)v=v$ forces, examining the $V_{n-1}$-component of the left side (which must vanish, as the right side has none), $L(1)v=0$.   

As $v\in V_n$ is the constant section only graded piece and is assigned weight $k=2n$,   matching \cite{BCFMN} and Proposition \ref{proptopexact}. 
The diagonal eigenvalue is $\lambda_n=j(\gamma_0,\tau_0)^{k-2n}=j(\gamma_0,\tau_0)^0=1$ automatically, there is no separate resonance condition to check, as for 
$M_0(\Gam)=\CC$ in the scalar theory. 
\eqref{eqellipticscalar} order of $\zeta$ subtlety only reenters as in Proposition \ref{propellipticstructure}, which degree $n_0$ within a larger $F_NV$ can serve as a resonant leading term. 
Thus,  $f(\tau_0)\in\ker L(1)$. In the constant-section setting actually used this paper (Proposition \ref{proptopexact}, \cite{BCFMN}), quasi-primarity is not only sufficient but necessary, because there is no lower conformal degree available within $F_NV=V_n$ to absorb the discrepancy that Proposition \ref{propellipticstructure} would otherwise allow. 
For a general $F_NV$-valued section with $N>n_0$, by contrast, Proposition \ref{propellipticstructure} shows this conclusion is too strong: the leading term need only satisfy the (typically nontrivial) resonance condition $\lambda_{n_0}=1$ appropriate to its degree $n_0<N$ with quasi-primarity emerging only as a special feature of the no-completion-available case illustrated here. 
\end{example}
%% 
%%%%%%%%%%%%%%%%%%%%%%%%%%%%%%%%%%%%%%%%%%%%%%%%%%%%%%%%%%%%%%%%%%%%%%%%%%%%%%%%
%% 
\section{Concluding remarks and open problems}
\label{secconclusion} 
We have shown that the bundle-theoretic approach of \cite{BCFMN} to vertex operator algebra valued modular forms on the elliptic modular curve extends to an arbitrary Fuchsian group 
$\Gam$ in two different regimes. 
The bundle $\mathcal V_X$ and its connection $\nabla$ (Theorems A, B) are available, for every Fuchsian group. 
 The richer algebraic theory built on the quasi-modular generator $E_2$, the doubled QVOA 
$Q(V,\Gam)$, the kernel theorem, the free-module description 
$M(\Gam)\otimes V^{(2)}\cong M(V,\Gam)$, and the resulting Hilbert series appear, 
but only when $\Gam$ has a cusp (Theorems C, D, E) constructed via Hecke classical analytic continuation (Section \ref{secE2}). 
 When $\Gam$ is cocompact of genus $g\ge 2$, the case of primary interest for higher genus geometry, this second layer of structure is obstructed.  No holomorphic weight $2$ quasi-automorphic generator exists (Theorem \ref{thmatiyah}) by an application of Atiyah theorem on holomorphic connections.
 We have shown, however, that this obstruction is completely understood at the level of dimensions. For torsion-free cocompact $\Gam$, Theorem \ref{thmexactdim} replaces the   inequality of Theorem \ref{thmbound} with an exact formula, in which the entire deficiency is concentrated at the top conformal degree and equals the failure of quasi-primarity in $V$.
 Conjecture \ref{conjexactdim} proposes the natural extension to elliptic points, of which we have proved the top-degree half (Proposition \ref{proptopexact}).
 We have also discussed, with a full structural account (Lemma \ref{lemellipticfiniteorder}, Proposition \ref{propellipticstructure}), the behavior of $V$-valued automorphic forms at elliptic points. 
The cusp/cocompact dichotomy for $E_2^\Gam$  (entirely invisible at level one where the unique cusp of $\Gam(1)$ supplies the escape from Atiyah obstruction) remains, in our view, the most interesting structural feature uncovered by passing to higher genus. 

Freeness of $Q(\Gam)=M(\Gam)[E_2^\Gam]$, on which most of Part I depends, 
is given by Lemma \ref{lemfreeness}, a complete proof derived from existence of one holomorphic depth $1$ generator. 
The existence (Lemma \ref{lemheckefact}) is cited, to the classical work of Selberg and Roelcke.
 The extension of the Atiyah obstruction to groups with elliptic points (Theorem \ref{thmatiyah}) does not use orbifold degree statements, but reduces elementarily, via Selberg lemma (Lemma \ref{lemselberg}), to the torsion-free case.
 The generic-resonance hypothesis of Proposition \ref{propellipticstructure} is stated,
 and its scope delimited, in Remark \ref{rmkellipticgeneric}.  
It is not needed anywhere in this paper except for the general-signature case of Conjecture \ref{conjexactdim} beyond its top-degree term, which remains open. 
 The quasi-vertex-operator-algebra structure of Theorem \ref{thmdoubledQVOA} now rests on a  derived Ramanujan-type identity (Lemma \ref{lemramanujan}) and a  verification that $[\theta,\delta]$ acts as a scalar (in the proof of Theorem \ref{thmdoubledQVOA}), instead of a dimension count.
 In this paper, we have tried to state at each theorem which of validity, torsion-free-only validity, or conjectural status applies, and to fix a single convention (Remark \ref{rmkomeganotation}) for the ordinary bundle $K_X$ versus its orbifold analogue.
%%

%%%%%%%%%%%%%%%%%%%%%%%%%%%%%%%%%%%%%%%%%%%%%%%%%%%%%%%%%%%%%%%%%%%%%%%%%%%%%%%%%
%%  
Several questions raised by this extension seem to us worth pursuing: 
\noindent\qquad\textit{Conjecture \ref{conjexactdim}.} 
Complete the proof of the exact dimension formula for cocompact $\Gam$ of general signature. As the discussion following Proposition \ref{proppartialconj} makes precise, this reduces to a single further ingredient: an orbifold (parabolic) Serre duality and Riemann-Roch theorem showing $H^1(X,\mathcal Q_j)=0$ for the relevant orbifold line bundles $\mathcal Q_j$  conformal degrees $j$ below the top.  
A parallel problem is the multiple-resonance case of Proposition   
\ref{propellipticstructure}, more than one resonant conformal degree at a single elliptic point within a single $F_NV$, which the proof of Conjecture \ref{conjexactdim} along these lines would need to confront once $k$ is large relative to the orders $m_1$, $\dots$, $m_r$.

\noindent\qquad\textit{Arithmetic Hecke theory.}
 Consider the $V$-valued Hecke eigensystems of Section \ref{sechecke} for Shimura curve groups $\Gam$ coming from quaternion orders 
(which have no cusps, thus this sits inside the cocompact theory of Section \ref{seccocompact}), 
 generalizing \cite{BCFMN}. 
This requires identifying the relevant analogue of the classical Eichler-Selberg trace formula for $\Gam$, and relating the resulting $L$-functions to those of automorphic forms on the associated quaternion algebra.

\noindent\qquad\textit{The real-analytic structure theory.}  
Independently of Conjecture \ref{conjexactdim}, make the real-analytic substitute built from the Maass-Shimura operator $\delta_k$ of Remark \ref{rmkmaassshimura}: develop the structure theory of the nearly holomorphic space $\widetilde Q(V,\Gam)$ (its QVOA-type structure, a lowering operator built from $1/4\pi\Im\tau$, and a real-analytic analogue of the $P$-isomorphism). Such a theory would give an independent, Maass-Shimura-theoretic route to Conjecture \ref{conjexactdim}, alongside the algebro-geometric route of Item 1.

\noindent\qquad\textit{The Teichm\"uller-space direction.}    
Making precise any relationship between the variation of the uniformizing projective connection over Teichm\"uller space, the Weil-Petersson geometry of $\mathcal T_{g,n}$, and the projectively flat KZB-type connection on the bundle of $V$-conformal blocks over $\mathcal M_{g,n}$ studied in \cite{TUY, FBZ} would be a substantial undertaking in its right, likely requiring tools (such as the Quillen metric on determinant lines, the comparison of sections versus coinvariants) well beyond the methods used in this paper. 

\noindent\qquad\textit{Degeneration to the boundary of moduli space.} 
 Extend the construction of Sections \ref{secunivfamily}-\ref{secteichforms} across the Deligne-Mumford boundary $\partial\mathcal M_{g,n}$, where $\Gam_t$ degenerates to a noded Fuchsian group. This should be controlled by the factorization properties of $V$ familiar from the theory of conformal blocks \cite{TUY, FBZ}, and would give a compactified version of the bundle $\mathcal M_k(V)$ constructed here.

The material presented in this paper is also useful in other areas of mathematical physics 
\cite{RSZ, kmmzz} and condenced matter 
physics \cite{volzub, z9, z8, z7, z4, z5, z3, z1}. 

\medskip
%%%%%%%%%%%%%%%%%%%%%%%%%%%%%%%%%%%%%%%%%%%%%%%%%%%%%%%%%%%%%%%%%%%%%%%
%% 
\noindent\textbf{Acknowledgments.} 
The author is supported by the
Institute of Mathematics, Academy of Sciences of the Czech Republic
(RVO 67985840). 

\medskip
\noindent\textbf{Data Availability.}
Data sharing is not applicable to this article as no datasets were generated
or analysed during the current study.

\medskip
\noindent\textbf{Declarations}

\medskip
\noindent\textbf{Conflict of interest.}
The author has no conflicts of interest to declare that are relevant to the
content of this article.
%%
%%%%%%%%%%%%%%%%%%%%%%%%%%%%%%%%%%%%%%%%%%%%%%%%%%%%%%%%%%%%%%%%%%%%%%%%%%%%%%%%%%%%%%
%%%%%%%%%%%%%%%%%%%%%%%%%%%%%%%%%%%%%%%%%%%%%%%%%%%%%%%%%%%%%%%%%%%%%%%%%%%%%%%%%%%%%%
%% 

\end{document}